%% file: main.tex
\documentclass[sn-mathphys,Numbered]{sn-jnl}
\input{shared}

\begin{document}

\title[Article Title]{Parallel-in-Time Solutions with Random Projection Neural Networks}

\author[1]{Marta M.~Betcke}\email{m.betcke@ucl.ac.uk}

\author[2]{Lisa Maria Kreusser}\email{lmk54@bath.ac.uk}

\author[3]{\fnm{Davide} \sur{Murari}}\email{davide.murari@ntnu.no}

\affil[1]{\orgdiv{Department of Computer Science}, \orgname{University College London}}

\affil[2]{\orgdiv{Department of Mathematical Sciences}, \orgname{University of Bath}}

\affil[3]{\orgdiv{Department of Mathematical Sciences}, \orgname{Norwegian University of Science and Technology}}

\abstract{This paper considers one of the fundamental parallel-in-time methods for the solution of ordinary differential equations, Parareal, and extends it by adopting a neural network as a coarse propagator. We provide a theoretical analysis of the convergence properties of the proposed algorithm and show its effectiveness for several examples, including Lorenz and Burgers' equations. In our numerical simulations, we further specialize the underpinning neural architecture to Random Projection Neural Networks (RPNNs), a $2-$layer neural network where the first layer weights are drawn at random rather than optimized. This restriction substantially increases the efficiency of fitting RPNN's weights in comparison to a standard feedforward network without negatively impacting the accuracy, as demonstrated in the SIR system example.}


\maketitle

\section{Introduction}\label{se:introduction}

In this paper, we consider initial value problems expressed as a system
of first-order ordinary differential equations (ODEs). This wide class of problems arises in many social and natural sciences applications, including semi-discretized, time-dependent partial differential equations. We express a generic system of such differential equations as
\begin{align}\label{eq:ode}
\begin{cases}
\xp\left(t\right) = \F\left(\x\left(t\right)\right)\in\mathbb{R}^d,\\
\x\left(0\right) = \x_0,
\end{cases}
\end{align}
which will be our reference problem. Here, $'$ denotes the derivative with respect to the time variable. To guarantee the existence and uniqueness of its solutions, we assume that $\F:\mathbb{R}^d\rightarrow\mathbb{R}^d$ is a Lipschitz-continuous vector field and $t\in [0,T]$ for some $T>0$. Solving an initial value problem like \eqref{eq:ode} analytically is generally not a possibility, and hence one needs to rely on numerical approximations to the solution curve $t\mapsto \x(t)$. Numerical techniques rely on introducing a time discretization $0<t_1<\cdots <t_N = T$ of the interval $[0,T]$, with steps $\Delta t_n = t_{n+1}-t_{n}$, and computing approximations $\x_n$ of the solution $\x(t_n)$ at the nodes $t_n$, i.e., $\x_n\approx \x(t_n)$. A popular and established option is provided by one-step methods, such as Runge--Kutta schemes, which relate $\x_{n+1}$ to $\x_n$ in terms of a map $\varphi^{\Delta t_{n}}_{\F}$ of the form $\x_{n+1}=\varphi^{\Delta t_{n}}_{\F}(\x_n)$. Collocation methods are a subset of Runge--Kutta methods \cite[Section II.7]{hairer1993solving} with particular relevance to this paper. These methods aim to approximate the solution on each interval $[t_n,t_{n+1}]$ with a real polynomial ${\tx}$ of a sufficiently high degree and coefficients in $\mathbb{R}^d$. The updated solution is then computed as $\x_{n+1}=\varphi^{\Delta t_{n}}_{\F}(\x_n)$ evaluating the polynomial at $t=t_{n+1}$ as $\x_{n+1}=\tx(t_{n+1})\approx \x(t_{n+1})$. To determine the coefficients of the polynomial $\tx(t)$, one needs to solve the system of algebraic equations $\dtx(t_{n,c})=\F(\tx(t_{n,c}))$ for a set of $C$ collocation points $t_n\leq t_{n,1}<t_{n,2}<\cdots <t_{n,C}\leq t_{n+1}$.

As initial value problems define causal processes, many time-stepping schemes are sequential by nature, in the sense that to compute $\x_{n+1}$, one has to compute $\x_n$ first. Nonetheless, multiple successful approaches such as Parareal \cite{LionsEtAl2001}, PFASST \cite{Emmett2012}, PARAEXP \cite{gander2013paraexp}, and MGRIT \cite{Falgout2014} have introduced some notion of parallel-in-time solution of initial value problems \eqref{eq:ode}, see for instance \cite{Gander2015} for an overview of existing methods. 

In this work, we build upon the Parareal algorithm \cite{LionsEtAl2001}. The speedup in Parareal is achieved by coupling a fine time step integrator with a coarse step integrator. In each iteration, the coarse integrator updates the initial conditions of initial value problems on time subintervals, which can be solved in parallel and only entail fine step time integration over a short time. The elegance and strong theoretical grounding of the idea (see \cite{gander2008nonlinear,Gander2007}, for instance) led to a number of variants of the Parareal algorithm, including combinations of Parareal with neural networks \cite{Lee2022,Ibrahim2023,jin2023learning}.

In recent years, solving differential equations with machine learning approaches gained in popularity; see, for instance, \cite{karniadakis2021a} for a review. For learned methods to become staple solvers, understanding their properties and ensuring they reproduce the qualitative behavior of the solutions is paramount.
The problem of convergence and generalization for neural network-based PDE solvers has been considered in \cite{mishra2022estimates, doumeche2023convergence, de2022error}, for instance. An analysis of the approximation properties of neural networks in the context of PDE solutions is provided in \cite{opschoor2020deep, kutyniok2022theoretical}. In the context of ODEs, there is an increasing interest in developing deep neural networks to learn time-stepping schemes unrestricted by constraints of the local Taylor series, including approaches based on flow maps \cite{Liu2022}, model order reduction \cite{REGAZZONI2019108852}, and spectral methods \cite{Lange2021}.

In the context of combining Parareal with neural networks, Parareal with a physics-informed neural network as a coarse propagator was suggested in \cite{Ibrahim2023}. In \cite{Lee2022}, the authors introduced a parallel (deep) neural network based on parallelizing the forward propagation following similar principles to those behind Parareal. In \cite{jin2023learning}, it was proposed to learn a coarse propagator by parameterizing its stability function and optimizing the associated weights to minimize an analytic convergence factor of the Parareal method for parabolic PDEs.

Neural networks are generally considered as a composition of parametric maps whose weights are all optimized so that a task of interest is solved with sufficient accuracy. The common choice of the optimization procedure is gradient-based algorithms, which start from a random set of initial weights and update them iteratively until the stopping criterion has been reached. A class of neural networks where some of the weights are not updated at all is often called Random Projection Neural Networks (RPNNs), sometimes also called Extreme Learning Machines (ELMs) \cite{huang2006extreme,huang2015trends,rahimi2008weighted}. Despite their seemingly reduced capability of approximating functions, these neural networks retain most of the approximation results of more conventional neural networks. For example, as derived in \cite[Theorem 2.1]{huang2006extreme}, RPNNs with two layers and $H$ hidden neurons, where only the last layer is optimized while all other weights are independently sampled from any interval according to any continuous probability distribution, can interpolate with probability one any set of $H$ distinct input-to-output pairs. Their expressivity properties, see e.g.\ \cite{Rahimi2008,rahimi2008weighted,gonon2023approximation}, make them suitable for the approximation of solutions of ODEs which were successfully considered in \cite{fabiani2023parsimonious,mortari2019high,schiassi2021physics,dwivedi2020physics,de2022physics}, yielding accurate approximations in a fraction of the training time when compared to more conventional networks.

\subsection{Contributions}
In this work, we build a hybrid numerical method based on the Parareal framework, where a RPNN constitutes the coarse time stepping scheme. We first derive an a-posteriori error estimate for general neural network-based solvers of ODEs. This theoretical result allows us to replace the coarse integrator of the Parareal method with a RPNN while preserving its convergence guarantees. The RPNNs are trained online during the Parareal iterations. There are several benefits to the proposed procedure. First, our hybrid approach comes with theoretical guarantees and allows us to solve a differential equation such that the produced solution is accurate to a certain degree. Additionally, using RPNNs rather than a more conventional neural network leads to a significant speedup in the algorithm without sacrificing its capabilities. Indeed, as we show for the SIR problem, using RPNNs leads to about half of the computational time of the other method, even without accounting for the offline training phase of the more conventional network. Further, we demonstrate the effectiveness of the proposed approach, together with the timings of the components of the algorithm, and apply it to several examples in Section \ref{se:experiments}.

\subsection{Outline}
The outline of the paper is as follows. We start with introducing the Parareal algorithm and its convergence properties in Section \ref{se:parallel}. Section \ref{se:theory} presents the theoretical derivation of an a-posteriori error estimate for neural network-based solvers. This result relies on a non-linear variation of the constants formula, also called the Gröbner-Alekseev Theorem. In Section \ref{se:elm}, we propose a hybrid algorithm combining the Parareal framework with the RPNN-based coarse propagator. We study the convergence properties of this hybrid algorithm in Section \ref{se:convergenceFixedPoint}. The effectiveness of the proposed method is tested in Section \ref{se:experiments} on the benchmark dynamical systems studied in \cite{gander2008nonlinear} with the addition of the SIR and ROBER problems. We conclude with the summary and analysis of the obtained results in Section \ref{se:conclusions}.

\section{Parareal method}\label{se:parallel}
This section introduces the Parareal algorithm \cite{LionsEtAl2001} and presents a convergence result needed for our derivations. 
\subsection{The method}
The Parareal algorithm builds on two one-step methods that we call $\fine,\coarse:\mathbb{R}^d\to\mathbb{R}^d$, denoting the fine and coarse integrators with timestep $\Delta t$, respectively. There are multiple options to design such maps, one being to use the same one-step method but with finer or coarser timesteps, e.g.,
\[
\fine := \coarse[\Delta t/M]\circ \cdots \circ\coarse[\Delta t/M] = \left(\coarse[\Delta t/M]\right)^M,\,\,M\in\mathbb{N}.
\]
This strategy motivates the subscripts of the two maps since these methods rely on a fine and a coarse mesh. Another option to define $\fine$ and $\coarse$ is to use methods of different orders, hence different levels of accuracy with the same timestep $\Delta t$. Regardless of how we define these two methods, the map $\fine$ is more expensive to evaluate than $\coarse$. The goal of the Parareal algorithm is to get an approximate solution $\{\x_n\}_{n=0}^N$ over the mesh $t_0=0<t_1<\cdots < t_N=T$, $\Delta t_n=t_{n+1}-t_n$, with the same degree of accuracy as the one obtained with $\fine[\Delta t_n]$ but in a shorter time. This is achieved by transforming \eqref{eq:ode} into a collection of initial value problems on a shorter time interval by using $\coarse[\Delta t_n]$. This zeroth iterate of the method consists of finding intermediate initial conditions $\x_{n}^0$ by integrating \eqref{eq:ode} with $\coarse[\Delta t_n]$ to get
\[
\x_0^0=\x_0,\,\,\x_{n+1}^0 = \coarse[\Delta t_n]\left(\x_n^0\right),\,\,n=0,\ldots,N-1,\,\,\Delta t_n = t_{n+1}-t_n,
\]
and define the $N$ initial value problems on the subintervals
\begin{align}\label{eq:para}
\begin{cases}
\xp\left(t\right) = \F\left(\x\left(t\right)\right), \\
\x\left(t_n\right) = \x_n^0,
\end{cases}\quad t\in \left[t_n,t_{n+1}\right],\,n=0,\ldots,N-1.
\end{align}
These problems can now be solved in parallel using the fine integrator $\fine[\Delta t_n]$, which constitutes the parallel step in all successive Parareal iterates. A predictor-corrector scheme is used to iteratively update the initial conditions on the subintervals $[t_n,t_{n+1}]$. Parareal iteration $i+1$ reads
\begin{align}\label{eq:correction}
\x_{n+1}^{i+1} = \fine[\Delta t_n]\left(\x_n^i\right) + \coarse[\Delta t_n]\left(\x_n^{i+1}\right) - \coarse[\Delta t_n]\left(\x_n^{i}\right),\,\,n=0,\ldots,N-1.
\end{align}
A common choice of a stopping criterion is $\max_{n=1,\ldots,N}\|\x_{n}^{i+1} - \x_n^i\|_2<\varepsilon$ for some tolerance $\varepsilon>0$. The parallel speedup is achieved if this criterion is met with far less iterates than the number of time intervals $N$.

\subsection{Interpretation of the correction term} 
Following \cite{gander2008nonlinear}, we provide the interpretation of \eqref{eq:correction} as an approximation of the Newton step for matching the exact flow at the time discretization points $t_0=0,\ldots,t_{N}=T$. We consider
\[
\mathcal{H}\left(\y\right) := \begin{bmatrix} 
\y_0 - \x_0\\
\y_1 - \phi^{\Delta t_{0}}_{\mathcal{F}}\left(\y_0\right) \\ 
\y_2 - \phi^{\Delta t_{1}}_{\mathcal{F}}\left(\y_1\right) \\ 
\vdots \\ 
\y_N - \phi^{\Delta t_{N-1}}_{\mathcal{F}}\left(\y_{N-1}\right) \end{bmatrix}=0,\quad \y=\begin{bmatrix}\y_0 \\ \y_1 \\ \vdots \\ \y_N \end{bmatrix}\in\mathbb{R}^{d\cdot (N+1)},
\]
where $\phi^{\Delta t}_{\mathcal{F}}(\x_{n})$ with $\x_n\in \mathbb R^d$ is the exact solution $\x(\Delta t)$ of the initial value problem
\[
\begin{cases}
\xp\left(t\right)=\F\left(\x\left(t\right)\right),\\
\x(0)=\x_n.
\end{cases}
\]
Linearizing $\mathcal H$ at the $i$th iterate, $\x^i$, equating it to 0 and solving for the $i+1$st iterate, $\x^{i+1}$, we arrive at the Newton update 
\[
\x_{n+1}^{i+1} = \phi^{\Delta t_{n}}_{\mathcal{F}}\left(\x_n^{i}\right) + \partial_{\x}\left(\phi^{\Delta t_{n}}_{\mathcal{F}}\right)\left(\x_n^i\right)\left(\x_n^{i+1}-\x_n^i\right), \quad n = 0,\dots,N-1
\]
for the solution of the system $\mathcal{H}(\y)=0$. The idea behind Parareal is then to approximate the unknown $\phi^{\Delta t_n}_{\mathcal{F}}(\x_n^i)$ with $\fine[\Delta t_n](\x_n^i)$, and the first order term with
\[
\partial_{\x}\left(\phi^{\Delta t_n}_{\mathcal{F}}\right)\left(\x_n^i\right)\left(\x_n^{i+1}-\x_n^i\right)\approx \coarse[\Delta t_n]\left(\x_n^{i+1}\right) - \coarse[\Delta t_n]\left(\x_n^{i}\right),
\]
which yields \eqref{eq:correction}.

\subsection{Convergence}
Convergence of the Parareal iterations was proven in \cite{gander2008nonlinear} under the assumption that the fine integrator $\fine[\Delta t]$ and the exact flow map $\phi_{\mathcal{F}}^{\Delta t}$ coincide.
\begin{theorem}[Theorem 1 in \cite{gander2008nonlinear}]\label{thm:convergenceParareal}
Let us consider the initial value problem \eqref{eq:ode} and partition the time interval $[0,T]$ into $N$ intervals of size $\Delta t = T/N$ using a grid of nodes $t_n=n\Delta t$. Assume that the fine integrator $\fine[\Delta t]$ coincides with the exact flow map $\phi_{\mathcal{F}}^{\Delta t}$, i.e.\ $\fine[\Delta t]=\phi_{\mathcal{F}}^{\Delta t}$. Furthermore, suppose that there exist $p\in \mathbb{N}$, a set of continuously differentiable functions $c_{p+1},c_{p+2},\cdots$, and $\alpha>0$ such that
\begin{align}\label{eq:cond1}
\begin{split}
\fine\left(\x\right)-\coarse\left(\x\right) &= c_{p+1}\left(\x\right)(\Delta t)^{p+1}+c_{p+2}\left(\x\right)(\Delta t)^{p+2}+\cdots, \quad \text{and}\\
\left\|\fine\left(\x\right)-\coarse\left(\x\right)\right\|_2&\leq \alpha(\Delta t)^{p+1}
\end{split}
\end{align}
for every $\x\in\mathbb{R}^d$, and also that there exists $\beta>0$ such that
\begin{align}\label{eq:cond2}
\left\|\coarse\left(\x\right)-\coarse\left(\y\right)\right\|_2 \leq \left(1+\beta\Delta t\right)\left\|\x-\y\right\|_2
\end{align}
for every $\x,\y\in\mathbb{R}^d$. Then there exists a positive constant $\gamma$ such that, at the $i-$th iterate of the Parareal method, the following bound holds
\[
\left\|\x(t_n) - \x_n^{{i}}\right\|_2 \leq \frac{\alpha}{\gamma}\frac{\left(\gamma(\Delta t)^{p+1}\right)^{{i}+1}}{({i}+1)!}\left(1+\beta\Delta t\right)^{n-{i}-1}\prod_{j=0}^{{i}}\left(n-j\right).
\]
\end{theorem}
This result guarantees that as the iteration progresses, the method provides an increasingly accurate solution. Furthermore, when $i=n$, the last product on the right-hand side vanishes, which corresponds to the worst-case scenario of the sequential solution, a.k.a.~at the $n$th iterate, the above idealized Parareal method replicates the analytical solution for the time subintervals up to $t_n$.

We take advantage of this convergence result in Section \ref{se:elm}, constructing the coarse propagator as a neural network satisfying the assumptions of Theorem \ref{thm:convergenceParareal}.

\section{A-posteriori error estimate for neural network-based solvers}\label{se:theory}
We aim to design a hybrid parallel-in-time solver for \eqref{eq:ode} based on the Parareal algorithm. This procedure consists of the Parareal iteration where the coarse propagator $\coarse$ is replaced by a neural network. In Section \ref{se:elm}, we will focus on a particular class of neural networks, called Random Projection Neural Networks (RPNNs). For now, however, we do not specify the structure of the neural network and define it as a map $\mathcal{N}_{\theta}:[0,\Delta t]\times\mathbb{R}^d\to\mathbb{R}^d$, parametrized by weights $\theta$, and satisfying the initial condition of the ODE, $\ELM{\theta,0,0,{\x_0}}=\x_0$. 

In the classical Parareal iteration, the coarse propagator $\coarse[\Delta t_n]$ is a map satisfying $\x(t_{n+1})\approx \coarse[\Delta t_n](\x(t_{n}))$, where $\x(t)$ solves \eqref{eq:ode}.
The coarse propagator balances the cost versus accuracy of the approximation, with the sweet spot yielding optimal parallel speedup. With this in mind, we design our replacement to be a continuous function of time and to allow longer steps than commonly taken by single-step numerical methods as employed by $\coarse[\Delta t_n]$. Motivated by collocation methods \cite[Chapter II.7]{hairer1993solving}, we choose the weights of the neural network $\mathcal{N}_{\theta}$ so that it satisfies the differential equation \eqref{eq:ode} at some collocation points in the interval $[0,\Delta t]$. More explicitly, given a set $\{t_{1},\ldots,t_{C}\}\subset [0,\Delta t]$, we look for a set of weights $\theta$ minimizing the loss function
\begin{equation}\label{eq:loss}
\mathcal{L}\left(\theta,\x_0\right):=\sum_{c=1}^C\left\|\dELM{\theta,t_c,0,{\x_0}} - \F\left(\ELM{\theta,t_c,0,{\x_0}}\right)\right\|_2^2.
\end{equation}
Consistent with our convention, in \eqref{eq:loss} $'$ denotes the time derivative, i.e., the derivative with respect to the first component.

In the following, we propose an error analysis for the approximate solution $\mathcal{N}_{\theta}$. This error analysis allows us to provide a-posteriori theoretical guarantees on both, the accuracy of the network $\ELM{\theta,\Delta t,0,{\x_0}}$ as a continuous approximation of the solution, as well as its potential as a replacement of  $\coarse(\x_0)$ while keeping intact the convergence guarantees of Parareal. We focus on a practical error estimate based on quadrature rules. For an, albeit less practical, alternative estimate based on defect control see Section \ref{se:convergenceDefect} of the supplementary material.

\begin{assumption}\label{ass:collocation}
Assume that the collocation points $\{t_{1},\ldots,t_{C}\}\subset [0,\Delta t]$, with $t_{1}<\cdots<t_{C}$, define a Lagrange quadrature rule exact up to order $p$ for some given $p\geq 1$, i.e., there is a set of weights $\rho_1,\ldots,\rho_C$ for which
\begin{equation}\label{eq:quadrature}
\int_{0}^{\Delta t} f\left(t\right)\dd{t} = \sum_{c=1}^C \rho_c f\left(t_{c}\right)=:\qr{f,0,{\Delta t}},\,\,\forall f\in \mathbb{P}_{p-1},
\end{equation}
where $\mathbb{P}_{p-1}$ is the set of real polynomials of degree $p-1$.
\end{assumption}
For a set of collocation points satisfying Assumption \ref{ass:collocation} and any scalar $p-$times continuously differentiable function $f\in\mathcal{C}^{p}(\mathbb{R},\mathbb{R})$, it holds \cite[Chapter 9]{quarteroni2006numerical}
\begin{equation}\label{eq:errorEstimateQuad}
\left|\qr{f,0,{\Delta t}}-\int_{0}^{\Delta t}f\left(t\right)\dd{t}\right|\leq \kappa (\Delta t)^{p+1}\max_{\xi\in [0,\Delta t]}\left|f^{(p)}\left(\xi\right)\right|,\,{\kappa }>0,
\end{equation}
where $f^{(p)}$ is the derivative of $f$ of order $p$.

We can now formulate a quadrature-based a-posteriori error estimate for the \emph{continuous} approximation $\ELM{\theta,t,0,{\x_0}}$ that only requires the defect to be sufficiently small at the collocation points.
\begin{theorem}[Quadrature-based a-posteriori error estimate]\label{thm:QuadAposteriori}
Let $\x(t)$ be the solution of the initial value problem \eqref{eq:ode} with $\F\in\mathcal{C}^{p+1}(\mathbb{R}^d,\mathbb{R}^d)$. Suppose that Assumption \ref{ass:collocation} on the $C$ collocation points $0\leq t_{1}<\cdots<t_{C}\leq \Delta t$ is satisfied and assume that $\ELM{\theta,\cdot,0,{\x_0}}:[0,\Delta t]\to\mathbb{R}^d$ is smooth and satisfies the collocation conditions up to some error of magnitude $\varepsilon$, i.e.
\begin{equation}\label{eq:residualEpsilon}
\left\|\dELM{\theta,t_c,0,{\x_0}} - \F\left(\ELM{\theta,t_c,0,{\x_0}}\right)\right\|_2\leq \varepsilon,\,\,c=1,\ldots,C.
\end{equation}
Then, there exist two constants $\alpha,\beta>0$ such that, for all $t\in [0,\Delta t]$, 
\begin{equation}\label{eq:aposteriori}
\left\|\x\left(t\right)-\ELM{\theta,t,0,{\x_0}}\right\|_2  \leq \alpha (\Delta t)^{p+1} + \beta\varepsilon.
\end{equation}
\end{theorem}
The proof of Theorem \ref{thm:QuadAposteriori} is based on the Gr{\"o}bner-Alekseev formula \cite[Theorem~14.5]{hairer1993solving} that we now state for completeness. 
\begin{theorem}[Gr{\"o}bner-Alekseev]\label{thm:alekseev}
For $\F\in\mathcal{C}^1(\mathbb{R}^d,\mathbb{R}^d)$ and $\mathcal{G}:\mathbb{R}^d\rightarrow\mathbb{R}^d$ consider the solutions $\x(t)$ and $\y(t)$ of the two ODEs
\[
\begin{cases}
    \xp\left(t\right) = \F\left(\x\left(t\right)\right),\\
    \x\left(0\right) = \x_0,
\end{cases}\quad \begin{cases}
    \yp\left(t\right) = \F\left(\y\left(t\right)\right) + \mathcal{G}\left(\y\left(t\right)\right),\\
    \y\left(0\right) = \x_0,
\end{cases}
\]
assuming they both have a unique solution. For any times $0\leq s\leq t$, let $\phi_{\F}^{s,t}(\y(s))$ be the exact solution of the initial value problem 
\[
\begin{cases}
\xp\left(t\right) = \F\left(\x\left(t\right)\right),\\
\x\left(s\right) = \y\left(s\right).
\end{cases}
\]
Then, for any $t\geq 0$, one has
\begin{equation}\label{eq:perturbation}
\y\left(t\right) = \x\left(t\right) + \int_{0}^t \frac{\partial \phi_{\F}^{s,t}\left(\x_0\right)}{\partial \x_0}\bigg\vert_{\x_0 = \y\left(s\right)}\mathcal{G}(\y\left(s\right))\dd{s}.
\end{equation}
\end{theorem}
We now prove the a-posteriori error estimate in Theorem \ref{thm:QuadAposteriori} using Theorem \ref{thm:alekseev}.
\begin{proof}[Proof of Theorem \ref{thm:QuadAposteriori}]
Let $\x(t)$ be the solution of the initial value problem \eqref{eq:ode}. Further note that $t\mapsto \ELM{\theta,t,0,{\x}}$ satisfies the initial value problem
\[
\begin{cases}
   \dELM{\theta,t,0,{\x_0}}= \F\left(\ELM{\theta,t,0,{\x_0}}\right) +  \left[\dELM{\theta,t,0,{\x_0}} - \F\left(\ELM{\theta,t,0,{\x_0}}\right)\right],\\
    \ELM{\theta,0,0,{\x_0}}= \x_0.
\end{cases}
\]
{Setting  $\mathcal{G}\left(\ELM{\theta,t,0,{\x_0}}\right) =\dELM{\theta,t,0,{\x_0}} - \F\left(\ELM{\theta,t,0,{\x_0}}\right)$, from \eqref{eq:perturbation} we obtain}
\begin{align*}
&\left\|\x\left(t\right)-\ELM{\theta,t,0,{\x_0}}\right\|_2 = \left\| \int_{0}^{t} \frac{\partial \phi_{\F}^{s,t}\left(\y_0\right)}{\partial \y_0}\bigg\vert_{\y_0 = \ELM{\theta,s,0,{\x_0}}}\mathcal{G}\left(\ELM{\theta,s,0,{\x_0}}\right)\dd{s} \right\|_2 \\
&\leq \delta \left\| \int_{0}^{t}\left[\dELM{\theta,s,0,{\x_0}} - \F\left(\ELM{\theta,s,0,{\x_0}}\right)\right]\dd{s} \right\|_2,
\end{align*}
where $0 < \delta  <\infty$ bounds the norm of the Jacobian matrix of $\phi_{\F}^{s,t}$ for $0\leq s\leq t\leq \Delta t$ by virtue of $\F\in\mathcal{C}^1(\mathbb{R}^d,\mathbb{R}^d)$. 
Approximating the integral with the quadrature and subsequently bounding the residual at the collocation points, we obtain
\begin{align*}
\left\|\x\left(t\right)-\ELM{\theta,t,0,{\x_0}}\right\|_2
&\leq \delta \left(\left\| \sum_{c=1}^C\rho_{c} \left[\dELM{\theta,t_c,0,{\x_0}} - \F\left(\ELM{\theta,t_c,0,{\x_0}}\right)\right] \right\|_2  + {\bar{\kappa}}(\Delta t)^{p+1}\right)\\
&\leq  \delta\left(\varepsilon\sum_{c=1}^C\left|\rho_c\right|+ {\bar{\kappa}}(\Delta t)^{p+1}\right),
\end{align*}
where $t\in [0,\Delta t]$ and
\[
{\bar{\kappa}:= \kappa }\cdot\left(\max_{t\in \left[0,\Delta t\right]}\left\|\frac{d^{p}}{dt^{p}}\left[\dELM{\theta,t,0,{\x_0}} - \F\left(\ELM{\theta,t,0,{\x_0}}\right)\right]\right\|_2\right)>0,
\]
the right-hand side of \eqref{eq:errorEstimateQuad}. To conclude the proof we set $\alpha = {\bar{\kappa}}\delta$, $\beta=\delta\sum_{c=1}^C\left|\rho_c\right|$.
\end{proof}
While for the proof it {suffices that $\delta$ is finite, more practical bounds based on the one-sided Lipschitz constant of the vector field can be obtained.} We derive such a bound in Section \ref{app:boundNormJac} of the supplementary material.

{Given $\beta\varepsilon\ll (\Delta t)^{p+1}$, Theorem \ref{thm:QuadAposteriori} implies that} the approximation provided by the neural network is as accurate as the one provided by a {$p$th order one-step method with step size $\Delta t$}. This result allows us to replace the coarse integrator $\coarse$ with a neural network-based solver maintaining the convergence properties of Parareal.

\section{Parareal method based on Random Projection Neural Networks}\label{se:elm}
The theoretical results presented in Section \ref{se:theory} hold for generic continuous approximate solutions, particularly those provided by any neural network $\mathcal{N}_{\theta}$. {We now restrict the neural architecture to Random Projection Neural Networks (RPNNs) which allow a more efficient, hence faster, solution of the  optimization problem \eqref{eq:loss}} as we will highlight in Section \ref{se:experiments}.
\subsection{Architecture design}
RPNNs are feedforward neural networks composed of two layers, with trainable parameters confined to the outermost layer.
We draw the weights of the first layer from the continuous uniform distribution {$\mathcal{U}(\underline{\omega},\overline{\omega})$, for a lower bound $\underline{\omega}$ and an upper bound $\overline{\omega}$ which are set to $-1$ and $1$, respectively, in all our experiments.} We then aim to approximate the solution of \eqref{eq:ode} at a time $t$ with the parametric function
\begin{align}
&\ELM{\theta,t,0,{\x_0}} =\x_0 + \sum_{h=1}^H\bb{w}_h\left(\varphi_h\left(t\right) - \varphi_h\left(0\right)\right) =\x_0 + \sum_{h=1}^H\bb{w}_h\left(\sigma\left(a_h t +b_h\right) - \sigma\left(b_h\right)\right)\nonumber\\
 &={\x_0 + \theta^{\top}\left(\sigma\left(\bb{a}\, t + \bb{b}\right)-\sigma\left(\bb{b}\right)\right),\,\,\bb{a}=\begin{bmatrix} a_1 \\ \vdots \\ a_H\end{bmatrix},\,\,\bb{b}=\begin{bmatrix} b_1 \\ \vdots \\ b_H\end{bmatrix} \in\mathbb{R}^H,\,\,\theta=\begin{bmatrix} \bb{w}_1^{\top} \\ \vdots\\ \bb{w}_H^{\top}\end{bmatrix}\in\mathbb{R}^{H\times d},}\label{eq:ELM}
\end{align}
by training the weights collected in the matrix $\theta$. Here, $\varphi_h(t)=\sigma(a_h t + b_h)\in \mathbb R$, $h=1,\ldots,H$, is a given set of basis functions with {$a_h,b_h\sim\mathcal{U}(\underline{\omega},\bar\omega)$,}
and $\sigma\in\mathcal{C}^{\infty}(\mathbb{R},\mathbb{R})$ a smooth activation function. In the numerical experiments, we always choose {$\sigma(\cdot)=\tanh(\cdot)$}. The architecture in \eqref{eq:ELM} satisfies the initial condition of \eqref{eq:ode}, i.e., $\ELM{\theta,0,0,{\x_0}}=\x_0$. 
In addition, $t\mapsto\ELM{\theta,t,0,{\x_0}}$ and $\sigma$ have the same regularity.

\subsection{Algorithm design}
Our method closely mimics the Parareal algorithm, but with the network \eqref{eq:ELM} deployed as a coarse propagator in the Parareal update \eqref{eq:correction}. While in the classical Parareal algorithm the coarse propagator $\coarse[\Delta t_n]$ is assumed to be known for all sub-intervals and Parareal iterations, we do not make this assumption in our approach. Instead, we determine individual weights for the update of each of the initial conditions featuring in the update \eqref{eq:correction} to allow for a better adaptation to the local behavior of the approximated solution. Furthermore, our neural coarse integrator $\coarse[\Delta t_n]$ is not known ahead of time but is recovered and changing in the course of the Parareal iteration. Learning the coarse integrator involves training a RPNN on each of the sub-intervals to solve the ODE \eqref{eq:ode} at a set of fixed collocation points in the sub-interval. This procedure would be prohibitively expensive for generic neural networks trained with gradient-based methods. However, for RPNNs, estimating the matrix $\theta$ in \eqref{eq:ELM} is considerably cheaper and {comparable with classical collocation approaches, striking a balance between the computational efficiency, desirable behavior, and flexibility of the integrator.} Finally, in Section \ref{se:convergenceFixedPoint} we demonstrate that our approach is provably convergent.

\subsection{{Training strategy}}
The neural coarse propagator for solution \eqref{eq:ode} on the time interval $[0,T]$ is obtained by splitting the interval into $N$ sub-intervals, $\Delta t_{n} = t_{n+1}-t_{n}$, $t_0=0<t_1<\cdots < t_N=T$, and training $N$ individual RPNNs in sequence. On the $n$th sub-interval $[t_n, t_{n+1}]$ we train a RPNN of the form \eqref{eq:ELM} to solve the ODE system \eqref{eq:ode} approximately on this sub-interval. The initial condition at time $t_{n}$ is obtained by the evaluation of the previous Parareal correction step. Since the vector field $\F$ in \eqref{eq:ode} does not explicitly depend on the time variable, we can restrict our presentation to a solution on a sub-interval $[0,\Delta t_n]$
\begin{align}\label{eq:odesubinterval}
\begin{cases}
\xp\left(t\right)=\F\left(\x\left(t\right)\right),\\
\x\left(0\right)=\x_n^i,
\end{cases}
\end{align}
where the superscript $i$ refers to $i$th Parareal iterate.

To train {a RPNN \eqref{eq:ELM} on the sub-interval} $[0,\Delta t_n]$, we introduce $C$ collocation points $0 \leq t_{n,1} < \cdots < t_{n,C} \leq \Delta t_n$, where the subscript $n$ keeps track of the interval length $\Delta t_n$ {emphasizing the independent choice of collocation points on each sub-interval}. For a given initial condition $\x_n^{{i}}$, we  find a matrix $\theta_n^{{i}}$ {such that $\mathcal N_{\theta_n^{{i}}}$ approximately satisfies \eqref{eq:odesubinterval}} for all $t_{n,c}$, $c=1,\ldots,C$, by solving the optimization problem
\begin{equation}\label{eq:optProblem}
\theta_n^i = \arg\min_{\theta\in\mathbb{R}^{H\times d}}\sum_{c=1}^C\left\|\dELM{\theta,t_{n,c},0,{\x_n^{{i}}}} - \F\left(\ELM{\theta,t_{n,c},0,{\x_n^{{i}}}}\right)\right\|_2^2.
\end{equation} 
{This hybrid Parareal method returns approximations of the solution at the time nodes $t_0,...,t_N$, which we call $\x_0,...,\x_N$. Furthermore, since the RPNNs on sub-intervals are smooth functions of time, one could also access a piecewise smooth approximation of the curve $[0,T]\ni t\mapsto \x(t)$ by evaluating the individual RPNNs upon convergence of the Parareal iteration}
\begin{equation}\label{eq:piecewise}
\tx\left(t\right) = \ELM{\theta_n,t-t_n,0,{\x_n}},\,\,t\in \left[t_n,t_{n+1}\right),\,\,n=0,\ldots,N-1.
\end{equation}
{Here, $\theta_n$ and $\x_n$ are the weight matrix and the initial condition at the time $t_n$ in the final Parareal iteration.} {Note that even though the points $\x_n^{{i}}$ are updated in each Parareal iteration \eqref{eq:correction}, they do not tend to change drastically, and we can initialize $\theta_{n}^{i+1}$ in \eqref{eq:optProblem} with the previous iterate $\theta_n^i$ to speedup convergence.} We terminate the Parareal iteration when either the maximum number of iterations is reached, or the difference between two consecutive iterates satisfies a given tolerance.
\begin{algorithm}[ht!]
\caption{Hybrid Parareal algorithm based on RPNNs}\label{alg:hybridParareal}
\begin{algorithmic}[1]
\State{\textbf{Inputs : }$\x_0$, \texttt{tol}, \texttt{max\_it}}
\State {$\texttt{error}\gets \texttt{tol}+1$, $i\gets 1$, $\x_0^0\gets\x_0$}
\For{$n = 0$ \textbf{to} $N-1$}\Comment{Zeroth iterate}\label{line:zeroth}
    \State \text{Find $\theta_n^0 = \arg\min_{\theta\in\mathbb{R}^{H\times d}}\sum_{c=1}^C\left\|\dELM{\theta,t_{n,c},0,{\x_n^0}} - \F\left(\ELM{\theta,t_{n,c},0,{\x_n^0}}\right)\right\|^2_2$}\label{line:opt_0}
    \State $\x_{n+1}^0\gets \ELM{\theta_n^0,\Delta t_{n},0,{\x_n^0}}$, $\x^{S,-1}_{n+1}\gets \ELM{\theta_n^0,\Delta t_{n},0,{\x_n^0}}$
\EndFor
\While{$i < \texttt{max\_it}$ \textbf{and} $\texttt{error} > \texttt{tol}$}
\State $\texttt{error}\gets 0$
\For{$n=0$ \textbf{to} $N-1$}\Comment{Fine integrator, \textbf{Parallel} For Loop}
    \State $\x_{n+1}^F \gets \fine[\Delta t_{n}](\x_n^{i-1})$\label{line:fine}
\EndFor
\State $\x_0^i\gets \x_0$
\For{$n=0$ \textbf{to} $N-1$}
    
    \State \text{Find $\theta_n^i = \arg\min_{\theta\in\mathbb{R}^{H\times d}}\sum_{c=1}^C\left\|\dELM{\theta,t_{n,c},0,{\x_n^i}} - \F\left(\ELM{\theta,t_{n,c},0,{\x_n^i}}\right)\right\|^2_2$}\label{line:opt_i}
    \State {$\x_{n+1}^S\gets \ELM{\theta_n^i,\Delta t_{n},0,{\x_n^i}}$}\Comment{Next coarse approximation}
    \State {$\x_{n+1}^i\gets \x_{n+1}^F + \x_{n+1}^S - \x^{S,-1}_{n+1}$}\Comment{Parareal correction}\label{line:correction}
    \State {$\x^{S,-1}_{n+1}\gets \x_{n+1}^S$}
    \State \texttt{error}$\gets \max\left\{\texttt{error},\left\|\x_{n+1}^i-\x_{n+1}^{i-1}\right\|_2\right\}$
\EndFor
\State $i\gets i+1$
\EndWhile
\State \Return $\left\{\x_0^{i-1},\ldots,\x_N^{i-1}\right\}$, $\left\{\theta_0^{i-1},\ldots,\theta_{N-1}^{i-1}\right\}$
\end{algorithmic}
\end{algorithm}
\subsection{Implementation details}
Our hybrid Parareal method is described in Algorithm \ref{alg:hybridParareal} and the  Python code can be found in the associated GitHub repository\footnote{\url{https://github.com/davidemurari/RPNNHybridParareal}}. The zeroth iterate of the method, starting in line \ref{line:zeroth}, only relies on RPNNs to get intermediate initial conditions $\x_n^0$, $n=0,\ldots,N-1$. These initial conditions are then used to solve with the fine integrators $\fine[\Delta t_n]$ the $N$ initial value problems in parallel, see line \ref{line:fine}. These approximations are subsequently updated in the Parareal correction step of line \ref{line:correction}.
 
The Algorithm \ref{alg:hybridParareal} relies on solving a \emph{nonlinear} optimization problem in lines \ref{line:opt_0} and \ref{line:opt_i} to update the weights $\theta_n^i$.
For all systems studied in Section \ref{se:experiments} but the Burgers' equation, we use the Levenberg–Marquardt algorithm \cite[Chapter 10]{nocedal1999numerical}. For Burgers' equation, we rely on the Trust Region Reflective algorithm \cite{branch1999subspace} to exploit the sparsity of the Jacobian matrix. The optimization algorithms are implemented with the $\texttt{least\_squares}$ method of $\texttt{scipy.optimize}$. In both cases, we provide an analytical expression of the Jacobian of the loss function with respect to the weight $\theta$, derived in Section \ref{app:jacobian} of the supplementary material. Additionally, all the systems but the ROBER problem are solved on a uniform grid, i.e., $t_n = nT/N$. For the ROBER problem, we work with a non-uniform grid, refined in $[0,1]$, to capture the spike in the solution occurring at small times.

As common in neural network-based approaches for solving differential equations, see, e.g.\ \cite{de2022physics}, we opt for $C$ equispaced collocation points in each time interval. We also tested Lobatto quadrature points in the Lorenz example in subsection \ref{se:lorenz}. In all experiments, we set $C=5$ and the number of hidden neurons $H=C=5$ to match.

\section{Convergence of the RPNN-based Parareal method}\label{se:convergenceFixedPoint}
In this section, we study the convergence properties of Algorithm \ref{alg:hybridParareal}. Following the Parareal analysis in Theorem \ref{thm:convergenceParareal} we only need to consider the time interval $[0,\Delta t]$ and collocation points $0<t_1<\cdots<t_C<\Delta t$ satisfying Assumption \ref{ass:collocation}. 

We write our solution ansatz, \eqref{eq:ELM}, and its time derivative evaluated at the collocation points as the matrices
\[
\tX_{\theta}\left(\x,\Delta t\right) = \begin{bmatrix} \ELM{\theta,t_{1},0,{\x}}^{\top} \\ \vdots\\ \ELM{\theta,t_{C},0,{\x}}^{\top} \end{bmatrix}\in \mathbb R^{C \times d},\,\quad\dtX_{\theta}\left(\x,\Delta t\right) = \begin{bmatrix} \dELM{\theta,t_{1},0,{\x}}^{\top} \\ \vdots\\ \dELM{\theta,t_{C},0,{\x}}^{\top} \end{bmatrix}\in \mathbb R^{C \times d},
\]
and shorthand the evaluation of the vector field $\F$ on the rows of the matrix $\bfX\in\mathbb{R}^{C\times d}$ 
\[
\bfF\left(\bfX\right) = \begin{bmatrix} \F\left(\bfX^{\top}\bb{e}_1\right)^{\top} \\ \vdots\\ \F\left(\bfX^{\top}\bb{e}_C\right)^{\top} \end{bmatrix}\in \mathbb R^{C \times d},
\]
with $\bb{e}_1,\ldots,\bb{e}_C\in\mathbb{R}^C$ the canonical basis of $\mathbb{R}^C$.

We further rewrite the ansatz as
$\tX_{\theta}\left(\x,\Delta t\right) = \boldsymbol{1}_C \x^\top + (\bfH-\bfH_0) \theta$, where $\boldsymbol{1}_C=\begin{bmatrix} 1 & \dots & 1\end{bmatrix}^{\top}\in\mathbb{R}^C$, 
\[
\bfH = \begin{bmatrix} \sigma\left(\bb{a}^{\top} t_{1} + \bb{b}^{\top}\right) \\
	\vdots \\ 
	\sigma\left(\bb{a}^{\top} t_{C} + \bb{b}^{\top}\right) \end{bmatrix}\in\mathbb{R}^{C\times H},\quad \bfH' = \begin{bmatrix} \sigma'\left(\bb{a}^{\top} t_{1} + \bb{b}^{\top}\right)\odot \bb{a}^{\top} \\ \vdots \\ \sigma'\left(\bb{a}^{\top} t_{C} + \bb{b}^{\top}\right)\odot \bb{a}^{\top}\end{bmatrix}\in\mathbb{R}^{C\times H},
\]
with $\bb{a}^{\top}=\begin{bmatrix} a_1 & \cdots & a_H\end{bmatrix}$, $\bb{b}^{\top}=\begin{bmatrix} b_1 & \cdots & b_H\end{bmatrix}$, and $\sigma\in\mathcal{C}^{\infty}(\mathbb{R},\mathbb{R})$ evaluated componentwise while $\odot$ denotes the componentwise product. As for the experiments, we restrict ourselves to the case $C=H$ 
for which one can prove, see \cite[Theorem 2.1]{huang2006extreme}, 
that with probability one $\bfH$ and $\bfH'$ are invertible for $\bfa,\bfb$ drawn from any continuous probability distribution. 
Finally, $\bfH_0=\boldsymbol{1}_C\sigma\left(\bfb^{\top}\right)\in\mathbb{R}^{C\times H}$ in $\tX_{\theta}\left(\x,\Delta t\right)$, accounts for the initial condition.
\begin{theorem}[Existence and regularity of the solution]\label{thm:existence}
For the loss function
\begin{equation}\label{eq:matrixEqn}
\mathcal{L}(\theta,\x):=\left\|\dtX_{\theta}\left(\x,\Delta t\right) - \bfF\left(\tX_{\theta}\left(\x,\Delta t\right)\right)\right\|_F^2
\end{equation}
with $\mathcal{N}_{\theta}$ in $\tX_{\theta}$ defined as in \eqref{eq:ELM}, $\sigma$ a smooth $1-$Lipschitz activation function, and a choice of step size
\begin{align}\label{eq:deltatcond}
\Delta t \in \left(0,\left(\left\|\inv\right\|_2\mathrm{Lip}(\F)\sqrt{C}\|\bfa\|_2\right)^{-1}\right),
\end{align}
there exists a unique Lipschitz continuous function $\mathbb{R}^d\ni\x\mapsto \theta(\x)\in\mathbb{R}^{H\times d}$ such that $\mathcal{L}(\theta(\x),\x)=0$ for every $\x\in\mathbb{R}^d$.
\end{theorem}
We remark that the loss function in \eqref{eq:matrixEqn} using the Frobenius norm $\| \cdot \|_F$ is a reformulation of \eqref{eq:loss} in a matrix form. We now prove Theorem \ref{thm:existence} using a parameterized version of Banach Contraction Theorem presented in \cite[Lemma 1.9]{FB-CTDS}.
\begin{proof}
The requirement $\mathcal{L}(\theta,\x)=0$ implies that the ansatz $\tX_{\theta}\left(\x,\Delta t\right) = \boldsymbol{1}_C \x^\top + (\bfH-\overline{\bfH}) \theta$ and its derivative, $\tX_{\theta}'\left(\x,\Delta t\right)= \bfH' \theta$, satisfy the ODE \eqref{eq:odesubinterval}, $\dtX_{\theta}\left(\x,\Delta t\right)=\bfF\left(\tX_{\theta}\left(\x,\Delta t\right)\right)$, which can be equivalently written as $\bfH'\theta = \bfF\left(\boldsymbol{1}_C\x^{\top} + \left(\bfH-\overline{\bfH}\right)\theta\right)$. We introduce the fixed point map 
\[
T\left(\theta,\x\right) = \inv\bfF\left(\boldsymbol{1}_C\x^\top + \left(\bfH-\overline{\bfH}\right)\theta\right)\in\mathbb{R}^{H\times d},
\]
and, when not differently specified, we denote with $\mathrm{Lip}(f)$ the Lipschitz constant of a Lipschitz continuous function $f$ with respect to the $\ell^2$ norm. Since  $\mathrm{Lip}(\sigma)\leq 1$, we have
\[
\left\|\bfH-\overline{\bfH}\right\|_F^2=\sum_{c=1}^C\left\|\sigma\left(\bfa^{\top} t_{c}+\bfb^{\top}\right)-\sigma\left(\bfb^{\top}\right)\right\|_2^2\leq C\|\bfa\|_2^2(\Delta t)^2
\]
as $t_c\in (0,\Delta t)$. Furthermore,
\[
\left\|\bfF(\bfX)-\bfF(\bfY)\right\|_F^2=\sum_{c=1}^C\left\|\F(\bfX^{\top}\e_c)-\F(\bfY^{\top}\e_c)\right\|_2^2\leq \mathrm{Lip}\left(\F\right)^2\left\|\bfX-\bfY\right\|_F^2\]
for any $\bfX,\bfY\in\mathbb{R}^{C\times d}$. Setting $\ell_{\theta}= \left\|\inv\right\|_2\mathrm{Lip}(\F)\sqrt{C}\|\bfa\|_2\Delta t$ we conclude that $T(\cdot,\x)$ is Lipschitz continuous with constant $\ell_{\theta}<1$ for $\Delta t$ satisfying \eqref{eq:deltatcond}, as
\[
\left\|T(\theta_2,\x)-T(\theta_1,\x)\right\|_F \leq \left\|\inv\right\|_2\mathrm{Lip}(\F)\sqrt{C}\|\bfa\|_2\Delta t\left\|\theta_2-\theta_1\right\|_F\label{eq:contraction}=\ell_{\theta}\left\|\theta_2-\theta_1\right\|_F\nonumber
\]
for any $\theta_1,\theta_2\in \mathbb R^{H\times d}$. We note that the $2-$norm of $\inv$ can be used since for any pair of matrices $\mathbf{A},\mathbf{B}$ of compatible dimensions, it holds $\|\mathbf{A}\mathbf{B}\|_F\leq \|\mathbf{A}\|_2\|\mathbf{B}\|_F$. Furthermore, $T(\theta,\cdot)$ is Lipschitz continuous with Lipschitz constant given by $\ell_{\x}= \left\|\inv\right\|_2\mathrm{Lip}\left(\F\right)\sqrt{C}$, since
\begin{align*}
\left\|T\left(\theta,\x_2\right)-T\left(\theta,\x_1\right)\right\|_F&\leq \left\|\inv\right\|_2\mathrm{Lip}\left(\F\right)\left\|\boldsymbol{1}_C\left(\x_2-\x_1\right)^{\top}\right\|_F, \quad \forall \x_1,\x_2 \in \mathbb R^d\\
&= \left\|\inv\right\|_2\mathrm{Lip}\left(\F\right)\sqrt{C} \left\|\x_2-\x_1\right\|_2 
= \ell_{\x}\left\|\x_2-\x_1\right\|_2.
\end{align*}
By \cite[Lemma 1.9]{FB-CTDS}, we can hence conclude that, provided $\Delta t$ satisfies \eqref{eq:deltatcond}, there is a well-defined Lipschitz continuous function $\theta:\mathbb{R}^d\to\mathbb{R}^{H\times d}$, with
\[
\mathrm{Lip}\left(\theta\right)\leq \frac{\ell_\x}{1-\ell_\theta} = \frac{\left\|\inv\right\|_2\mathrm{Lip}\left(\F\right)\sqrt{C}}{1-\left\|\inv\right\|_2\mathrm{Lip}(\F)\sqrt{C}\|\bfa\|_2\Delta t},
\]
such that $\theta(\x)=T(\theta(\x),\x)$ for every $\x\in\mathbb{R}^{d}$, or equivalently $\mathcal{L}(\theta(\x),\x)=0$.
\end{proof}

\begin{proposition}[Convergence of the hybrid Parareal method]\label{pr:convergenceHybrid}
Consider the initial value problem \eqref{eq:ode} with $\F:\R^d\to\R^d$ a smooth Lipschitz continuous vector field. Suppose that the time interval $[0,T]$ is partitioned into $N$ intervals of size $\Delta t = T/N$ such that $\Delta t$ satisfies \eqref{eq:deltatcond} and choose the $C$ collocation points $0\leq t_{1}<\cdots<t_{C}\leq \Delta t$ to satisfy the Assumption \ref{ass:collocation}. Let $\sigma$ be a smooth $1-$Lipschitz function. Then for the coarse integrator $\coarse[\Delta t](\x) =\x+\theta(\x)^T\left(\sigma(\bfa \Delta t+\bfb)-\sigma(\bfb)\right)$ with $\theta(\x)$ as in Theorem \ref{thm:existence} and the fine integrator $\fine[\Delta t]=\phi_{\mathcal{F}}^{\Delta t}$, there exist positive constants $\alpha,\gamma,\beta$ such that, at the $i-$th iterate of the hybrid Parareal method, the following bound holds
\begin{equation}\label{eq:boundConvergenceHybrid}
\left\|\x(t_n) - \x_n^{{i}}\right\|_2 \leq \frac{\alpha}{\gamma}\frac{\left(\gamma(\Delta t)^{p+1}\right)^{{i}+1}}{({i}+1)!}\left(1+\beta\Delta t\right)^{n-{i}-1}\prod_{j=0}^{{i}}\left(n-j\right).
\end{equation}
\end{proposition}
\begin{proof}
The proof is based on showing that our network satisfies assumptions \eqref{eq:cond1} and \eqref{eq:cond2} of Theorem \ref{thm:convergenceParareal}. Theorem \ref{thm:existence} guarantees that, for $\Delta t$ satisfying \eqref{eq:deltatcond}, $\x \mapsto \theta(\x)$  is Lipschitz continuous with Lipschitz constant $\mathrm{Lip}\left(\theta\right)$. Further noting that $\left\|\sigma(\bb{a} \Delta t + \bb{b})-\sigma(\bb{b})\right\|_2 \leq \|\bfa\|_2\Delta t$ as $\mathrm{Lip}(\sigma)\leq 1$ we can write
\begin{align*}
\left\|\coarse\left(\x_2\right) - \coarse\left(\x_1\right)\right\|_2 &\leq 
\left\|\x_2-\x_1\right\|_2 + \|\bfa\|_2\,\Delta t\, \mathrm{Lip}\left(\theta\right)\left\| \x_2-\x_1\right\|_2\\
&=\left(1+\beta\Delta t\right)\left\|\x_2-\x_1\right\|_2
\end{align*}
for any $\x_1,\x_2 \in \mathbb R^d$, where $\beta:=\mathrm{Lip}\left(\theta\right)\|\bfa\|_2$, and hence condition \eqref{eq:cond2} is satisfied. Given that $\theta(\x)=T(\theta(\x),\x)$, as guaranteed by Theorem \ref{thm:existence}, satisfies the collocation conditions exactly, Theorem \ref{thm:QuadAposteriori} ensures that there exists $\alpha >0$ such that $\left\|\fine(\x)-\coarse(\x)\right\|_2\leq \alpha \left(\Delta t\right)^{p+1}$. Because of the smoothness of $\F$ and $\sigma$, one can also Taylor expand in time and guarantee the existence of continuously differentiable functions $c_{p+1},c_{p+2},\ldots$ such that
\[
\phi^{\Delta t}_{\F}(\x)-\ELM{\theta,\Delta t,0,{\x}}=c_{p+1}(\x)(\Delta t)^{p+1}+c_{p+2}(\x)(\Delta t)^{p+2}+\cdots .
\]
This allows concluding that \eqref{eq:cond1} holds, and the hybrid Parareal satisfies \eqref{eq:boundConvergenceHybrid}.
\end{proof}

As for the classical Parareal method, at the $n$th iterate, our hybrid Parareal method with the exact fine integrator replicates the analytical solution at the time instants $t_0,\ldots,t_n$. 

{In practice, as presented in the previous section, we do not have access to the function $\x\mapsto \theta(\x)$, but we only approximate its value at the points involved in the hybrid  Parareal iterates, i.e., $\theta_n^i\approx \theta(\x_n^i)$. Let us denote by $\hat{\theta}:\R^d\to\R^{H\times d}$ the function approximating $\theta$, so that $\theta_n^i=\hat{\theta}(\x_n^i)$. This function is typically provided by a convergent iterative method minimizing \eqref{eq:matrixEqn}.} Under the smoothness assumptions of Proposition \ref{pr:convergenceHybrid} and supposing the map $x\mapsto \hat{\theta}(x)$ is Lipschitz continuous, i.e., the adopted optimization method depends regularly on the parameter $\x\in\mathbb{R}^d$, the convergence in Proposition \ref{pr:convergenceHybrid} also holds for the approximate case. To see this, note that condition \eqref{eq:cond1} also holds for the approximate case as long as $\F$ is smooth enough and the collocation conditions are solved sufficiently accurately. In practice, based on \eqref{eq:residualEpsilon}, it suffices to have
\[
\max_{c=1,\ldots,C}\left\|\left(\dtX_{\theta_n^i}\left(\x,\Delta t\right) - \bfF\left(\tX_{\theta_n^i}\left(\x,\Delta t\right)\right)\right)^{\top}\e_c\right\|_2\leq \tilde{\alpha}\left(\Delta t\right)^{p+1}
\]
for an $\tilde{\alpha}>0$, and every $n=0,\ldots,N-1$ and iterate $i$. Furthermore, assumption \eqref{eq:cond2} follows from the Lipschitz regularity of the approximate function $\hat{\theta}$. 

\section{Numerical results}\label{se:experiments}
This section collects several numerical tests that support our theoretical derivations. We consider six dynamical systems, four of which come from the experimental section in \cite{gander2008nonlinear}, to which we add the SIR model and the ROBER problem. We assume that, for each of these systems, a single initial value problem is of interest and explore how RPNN-based coarse propagators perform for that initial value problem. For the one-dimensional Burgers' equation, we consider a semi-discretization with centered finite differences and provide the experimental results for different initial conditions, imposing homogeneous Dirichlet boundary conditions on the domain $[0,1]$.

The chosen fine integrators are classical Runge--Kutta methods with a smaller timestep than the coarse one $\Delta t$.  More explicitly, we assume that the coarse timestep $\Delta t$ is a multiple of the fine timestep $\delta t$ {and one coarse integrator step $\Delta t$, corresponds to $\Delta t/\delta t$ steps of the size $\delta t$ of the fine integrator. In all experiments, we use equispaced time collocation points, and for the Lorenz system, we also use Lobatto points. For stiff problems such as Burgers' and ROBER's, we use the implicit Euler method (IE), with update $\x_{n+1} = \x_n + \delta t \F(\x_{n+1})$, as a fine integrator, while for the others we found Runge-Kutta (RK4),
\[
\x_{n+1}= \x_n + \frac{\delta t}{6}\left(k_1+2k_2+2k_3+k_4\right)
\] 
with
\[
k_1=\F\left(\x_n\right),\,\, k_2=\F\left(\x_n+\delta t \frac{k_1}{2}\right),\,\, k_3 = \F\left(\x_n+\delta t \frac{k_2}{2}\right),\,\, k_4=\F\left(\x_n+\delta t k_3\right).
\]
to provide accurate solutions with moderately small step sizes. We specify the adopted timesteps in the dedicated sections below. 

The purpose of this paper is to demonstrate that our hybrid Parareal method based on RPNNs is theoretically motivated and practically effective, rather than the high-performance implementation. Thus, most of our experiments are run on a single processor where the parallel speedup would result from parallel execution of the fine integrators on the sub-intervals, in proportion to the number of cores used. To demonstrate the principle in hardware we run the ROBER's problem on five processors available to us and compare to the serial application of the fine integrator, however Parareal benefits will scale up with the problem size and number of cores. For Burgers' equation, we again use five processors for convenience since this allows us to do 100 repeated experiments faster.

In all plots, the label \say{para} refers to the hybrid methodology with neural networks as coarse propagators, while \say{ref} to the reference solution, obtained by the sequential application of the fine solver. We always plot the piecewise smooth Hybrid Parareal approximant constructed as \eqref{eq:piecewise}. We run the Hybrid Parareal until the difference between two consecutive iterates was at most $\texttt{tol}=10^{-4}$. As a safeguard, we put a hard limit, $\texttt{max\_it}=20$, on the iteration number, which was, however, not triggered in any of our experiments. All experiments were run on a MacBook Pro 2020 with Intel i5 processor and all the computational times were averaged over 100 runs per experiment. 
For each experiment, we report an average time per update of the coarse integrator on a sub-interval, which is also averaged over the number of sub-intervals. We measure the timing when computing the zeroth iterate in lines 3-6 of Algorithm \ref{alg:hybridParareal} to isolate the effects of warm starts used in Parareal update in later iterations. We also report a total average time to compute the solution, including the above mentioned coarse integrator updates along with possibly parallel execution of the fine step integrators.
\subsection{SIR}
The SIR model is one of the simplest systems considered in mathematical biology to describe the spread of viral infections. {SIR consists of three coupled ODEs for $\x = [x_1,x_2,x_3]^{\top}$ with parameters $\beta = 0.1$, and $\gamma = 0.1$:}
\begin{equation}\label{eq:sir}
\begin{cases}
x'_1\left(t\right) = -\beta x_1\left(t\right)x_2\left(t\right),\\
x'_2\left(t\right) = \beta x_1\left(t\right)x_2\left(t\right)-\gamma x_2\left(t\right),\\
x'_3\left(t\right) = \gamma x_2\left(t\right),\\
\x\left(0\right)=\begin{bmatrix}0.3 & 0.5 & 0.2\end{bmatrix}^{\top}.
\end{cases}
\end{equation}
\begin{figure}[ht!]
\centering
\includegraphics[width=.65\textwidth]{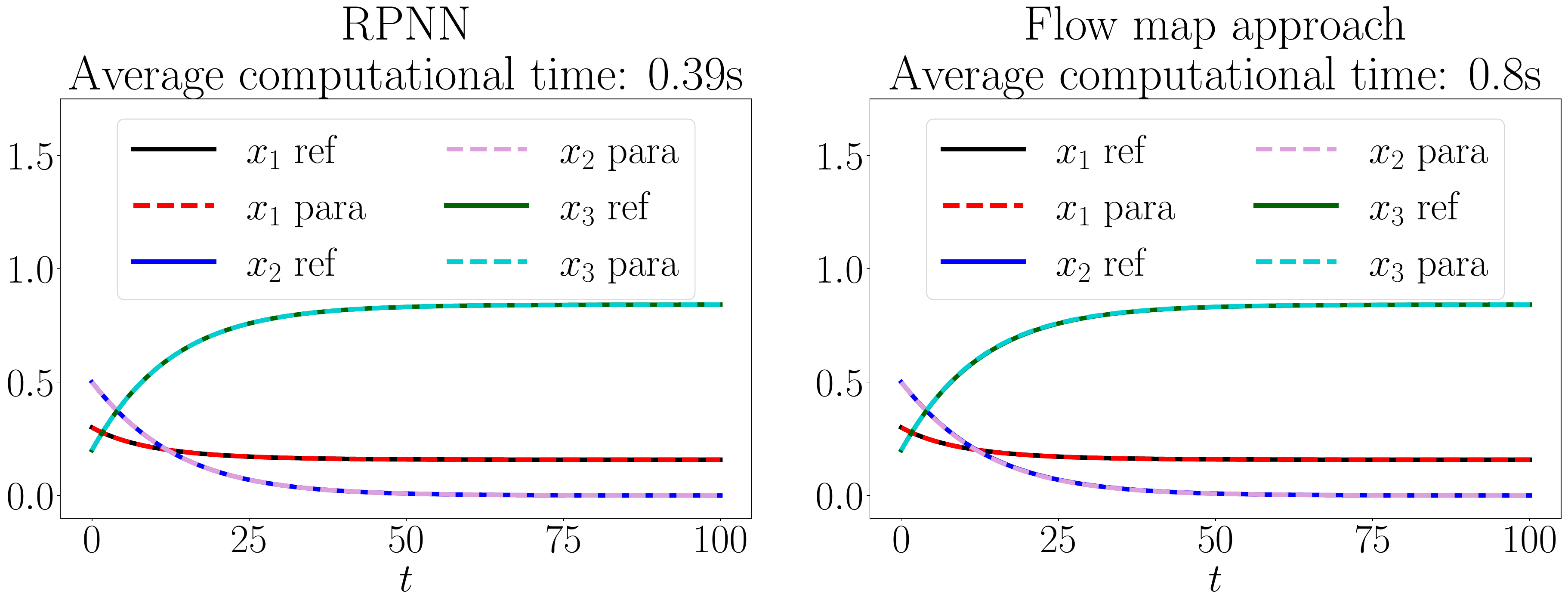}
\caption{SIR: Hybrid Parareal solution with (left) a RPNN-based coarse propagator, (right) flow map coarse propagator.}
\label{fig:paraSIR}
\vspace{-\baselineskip}
\end{figure}
We use this example to compare two different types of coarse propagators, the RPNN-based approach with a neural operator-type \textit{flow map} trained to approximate the solutions of the dynamical system described by \eqref{eq:sir} for initial conditions in the compact set $\Omega = [0,1]^3$, and times in $[0,1]$, see also \cite{flamant2020solving,wang2023long}.
Given that the Parareal method needs to evaluate the coarse propagator on several initial conditions, the learned flow map is the most natural neural network-based alternative, while a standard  Physics Informed Neural Network, which needs to be fitted for each initial condition, would be computationally too expensive. {Both coarse propagators use the same coarse timestep $\Delta t= 1$, while the fine solver timestep is $\delta t= 10^{-2}$.} {The piecewise smooth approximations computed with both methods are plotted in Figure \ref{fig:paraSIR}. We report the corresponding timings in Table \ref{tab:timeSIR}.}
\begin{table}[ht!]
    \centering
    \begin{tabular}{|l||c|c|}
    \hline
     Timing breakdown & RPNN & Flow \\
         \hline\hline
         Offline training phase & 0s & $\sim$20 minutes \\
         \hline
         Average cost coarse step in the zeroth iterate & 0.0009773s & 0.0002729s\\
         \hline
         Total  &  0.3940s &  0.8047s\\
         \hline
    \end{tabular}
    \caption{{SIR: Computational time for the RPNN and flow map based Hybrid Parareal variants on a single core.}}
    \label{tab:timeSIR}
    \vspace{-\baselineskip}
\end{table}

The RPNN-based approach took an average of 0.3940 seconds to {converge to} the final solution over 100 repeated experiments, while the flow map approach took an average time of 0.8047 seconds. The reason behind the higher cost of the flow map approach is that RPNNs minimize the residual more accurately than the flow map approach since they are trained for the specific initial conditions of interest, leading to a faster convergence of the Parareal method. If the offline training phase is accounted for, about 20 more minutes must be considered for the flow map approach, while no offline training is required for the RPNN-based approach. The offline training cost depends on the chosen architecture and training strategy. These details are provided in Section \ref{se:FlowMapApproach} of the supplementary material.

Given the reported results, it is clear that while both methods are comparable in terms of accuracy, the distribution of the costs is considerably different. The flow map approach has a high training cost and a low evaluation cost {but is also less accurate hence needing more Parareal iterations. On the other hand, the RPNN strategy, having no offline training phase and yielding more accurate solutions and hence needing fewer Parareal iterations}, saves substantial time. For this reason, we will only focus on the RPNN-based approach in the following experiments.

\subsection{ROBER}
The ROBER problem is a prototypical stiff system of coupled ODEs with parameters $k_1 = 0.04$, $k_2=3\cdot 10^7$, and $k_3 = 10^4$, 
\begin{equation}
\begin{cases}
x_1'\left(t\right) = -k_1 x_1\left(t\right)+k_3x_2\left(t\right)x_3\left(t\right),\\
x_2'\left(t\right) = k_1x_1\left(t\right)-k_2x_2^2\left(t\right)-k_3x_2\left(t\right)x_3\left(t\right),\\ x_3'\left(t\right) = k_2x_2^2\left(t\right),\\
\x\left(0\right) = \begin{bmatrix} 1 & 0 & 0 \end{bmatrix}^{\top}.
\end{cases}
\end{equation}

{As ROBER's solution spikes for short times, the usual approach is to discretize the time non-uniformly. Therefore we choose the coarse step size to be $\Delta t = 10^{-2}$ for times in $[0,1]$ and $\Delta t=3$ for times in $[1,100]$. The fine integrator timestep is $\delta t= 10^{-4}$.} We remark that ROBER's problem is commonly solved using a variable step-size method, for example, based on an embedded Runge-Kutta method \cite[Section II.4]{hairer1993solving}. {Fixing the step size allows us to understand how} the proposed hybrid method performs on stiff equations {without extra complication of step adaptivity. A variable step Parareal method (regardless if the coarse propagator is learned or classical), would involve adaptivity in both coarse and fine step and is beyond scope of this work.}
\begin{table}[ht!]
    \centering
  \begin{tabular}{|l || c| c|} 
    \hline
         Timing breakdown & RPNN  & Sequential IE, $\delta t$\\
         \hline\hline
         Average cost coarse step in the zeroth iterate & 0.001881s & \\
         \hline
         Average cost to produce the solution & 179.8280s &  263.2613s\\
         \hline
    \end{tabular}
    \caption{ROBER: Computational time for Hybrid Parareal using five cores versus sequential application of IE with fine step $\delta t$.}
    \label{tab:timeROBER}
    \vspace{-\baselineskip}
\end{table}
\begin{figure}[ht!]
\centering
\includegraphics[width=.5\textwidth]{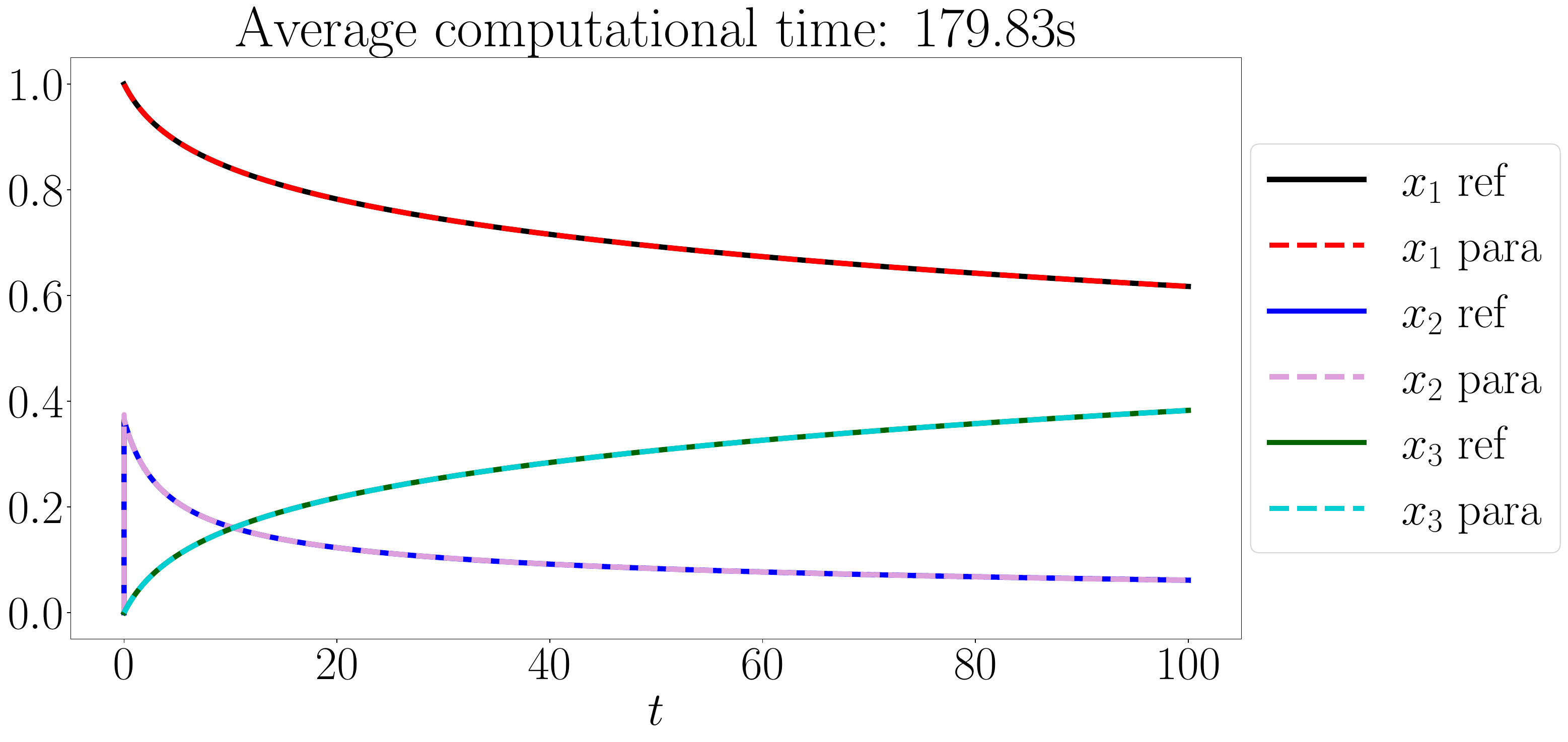}
\vspace{-\baselineskip}
\caption{{ROBER: Components of the
Hybrid Parareal solution. To plot all components on the same scale, $\x_2$ was scaled by a factor of $10^4$.}}
\label{fig:paraROBER}
\vspace{-\baselineskip}
\end{figure}
We report the obtained approximate solutions in Figure \ref{fig:paraROBER} and the timings in Table~\ref{tab:timeROBER}. {In these experiments, the fine integrators were executed in parallel on five cores. Thus, the total average time to compute the solution reflects the parallel speed up, albeit for a small number of cores. Given this stiff problem requires an implicit fine integrator, we expect the computational costs of the update of the coarse integrator, $0.001881$s, and one step of the fine integrator, $0.000263$s, to be closer than when using an explicit scheme as it is the case in the remaining examples. Additionally, to cover one coarse step, the fine integrator needs to perform at least 100 steps, given our choices for $\delta t$ and $\Delta t$. These respective costs help to optimally balance the choice of the number of sub-intervals versus the number of fine steps in each sub-interval, along with practical considerations like the number of cores available.}

\subsection{Lorenz}\label{se:lorenz}
For weather forecasts, real-time predictions are paramount, rendering parallel-in-time solvers highly relevant in this context. Lorenz's equations
\begin{equation}
\begin{cases}
x_1'\left(t\right)=-\sigma x_1\left(t\right) +\sigma x_2\left(t\right),\\
x_2'\left(t\right)=-x_1\left(t\right)x_3\left(t\right)+rx_1\left(t\right)-x_2\left(t\right),\\
x_3'\left(t\right)= x_1\left(t\right)x_2\left(t\right)-bx_3\left(t\right),\\
\x\left(0\right)=\begin{bmatrix} 20 & 5 & -5 \end{bmatrix}^{\top},
\end{cases} 
\end{equation}
describe one simple model for weather prediction.
Different parameter values give rise to considerably different trajectories for this system. We set $\sigma=10$, $r=28$, and $b=8/3$ to have chaotic behavior. We compute an approximate solution up to time $T=10$, using RPNNs as a coarse propagator with $\Delta T=T/250$ and RK4 with step $\delta t=T/14500$ as a fine integrator.
\begin{figure}[ht!]
\centering
\includegraphics[width=.65\textwidth]{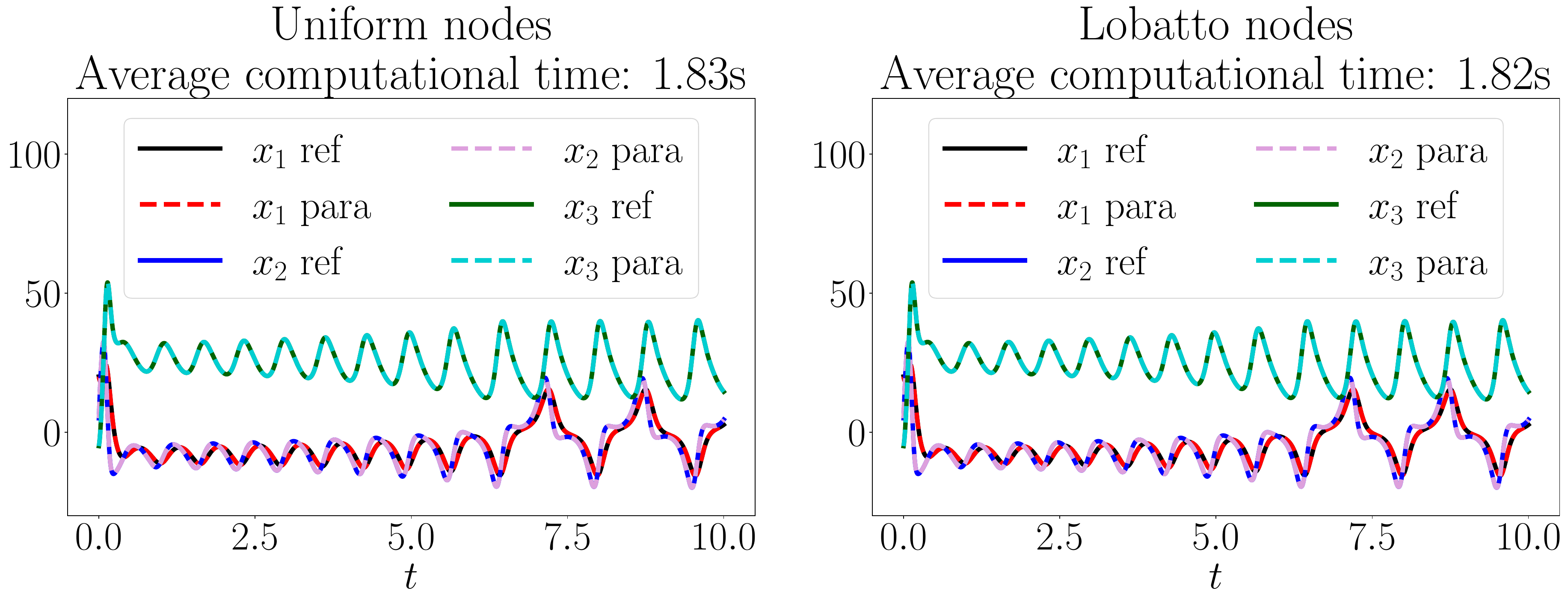}
\caption{Lorenz: Hybrid Parareal solution with (left) uniform collocation points, (right) Lobatto collocation points.}
\label{fig:paraLorenz}
\vspace{-\baselineskip}
\end{figure}

To show that the algorithm is not overly sensitive to the choice of the collocation points, we repeated the simulations using Lobatto collocation points. The qualitative behavior of the produced solutions for one choice of trained weights is reported in Figure \ref{fig:paraLorenz} and the corresponding timings} in Table \ref{tab:timeLorenz}. Although the Lorenz system is chaotic, the proposed hybrid solver provides an accurate approximate solution on the considered interval. Additionally, the average cost of one evaluation of the coarse RPNN-based integrator does not appear to depend strongly on the system's complexity but mostly on its dimension. Indeed, the average cost of one RPNN evaluation is comparable with the one for the SIR problem, see Table~\ref{tab:timeSIR}.
\begin{table}[ht!]
    \centering
    \begin{tabular}{|l ||c|c|}
    \hline
         Timing breakdown  & Uniform & Lobatto \\
         \hline\hline
         Average cost coarse step in the zeroth iterate &0.0009430s  & 0.0009371s\\
         \hline
         Average cost to produce the solution & 1.8312s & 1.8184s \\
         \hline
    \end{tabular}
    \caption{Lorenz: Computational time for the RPNN-based Hybrid Parareal with uniform and Lobatto nodes on a single core.}
    \label{tab:timeLorenz}
    \vspace{-\baselineskip}
\end{table}
\subsection{Arenstorf orbit}
The three-body problem is a well-known problem in physics that pertains to the time evolution of three bodies interacting because of their gravitational forces. Changing the ratios between the masses, their initial conditions, and velocities, can starkly alter the system's time evolution, and many configurations have been thoroughly studied. One of them is the stable Arenstorf orbit, which arises when one of the masses is negligible and the other two masses orbit in a plane. The equations of motion for this specific instance of the three-body problem are
\begin{equation}\label{eq:arenstorf}
\begin{cases}
x^{''}_1\left(t\right) = x_1\left(t\right)+2x_2'\left(t\right)-b\frac{x_1+a}{D_1}-a\frac{x_1'\left(t\right)-b}{D_2},\\
x^{''}_2\left(t\right) = x_2\left(t\right)-2x_1'\left(t\right)-b\frac{x_2(t)}{D_1}-a\frac{x_2(t)}{D_2},\\
\begin{bmatrix} x_1\left(0\right) & x_1'\left(0\right) & x_2\left(0\right) & x_2'\left(0\right)\end{bmatrix}^{\top} =\begin{bmatrix} 0.994& 0 & 0  & v_2^0 \end{bmatrix}^{\top},
\end{cases} 
\end{equation}
\[
D_1 = \left(\left(x_1\left(t\right)+a\right)^2+x_2\left(t\right)^2\right)^{3/2},\quad
D_2 = \left(\left(x_1\left(t\right)-b\right)^2+x_2\left(t\right)^2\right)^{3/2},
\]
$v_2^0=-2.00158510637908252240537862224$, $a=0.12277471$, and $b=1-a$. This configuration leads to a periodic orbit of period $17.06521656015796$ \cite{hairer1993solving}. In practice, we transform \eqref{eq:arenstorf} into a first order system via the velocity variables $v_1(t):=x_1'(t)$ and $v_2(t):=x_2'(t)$. We include the plot of the obtained solution for time up to  $T=17$ and timesteps $\Delta t=T/125$, and $\delta t = T/80000$, in Figure \ref{fig:paraArenstorf} and the timings in Table~\ref{tab:timeArenstorf}.
\begin{figure}[ht!]
\centering
\includegraphics[width=.65\textwidth]{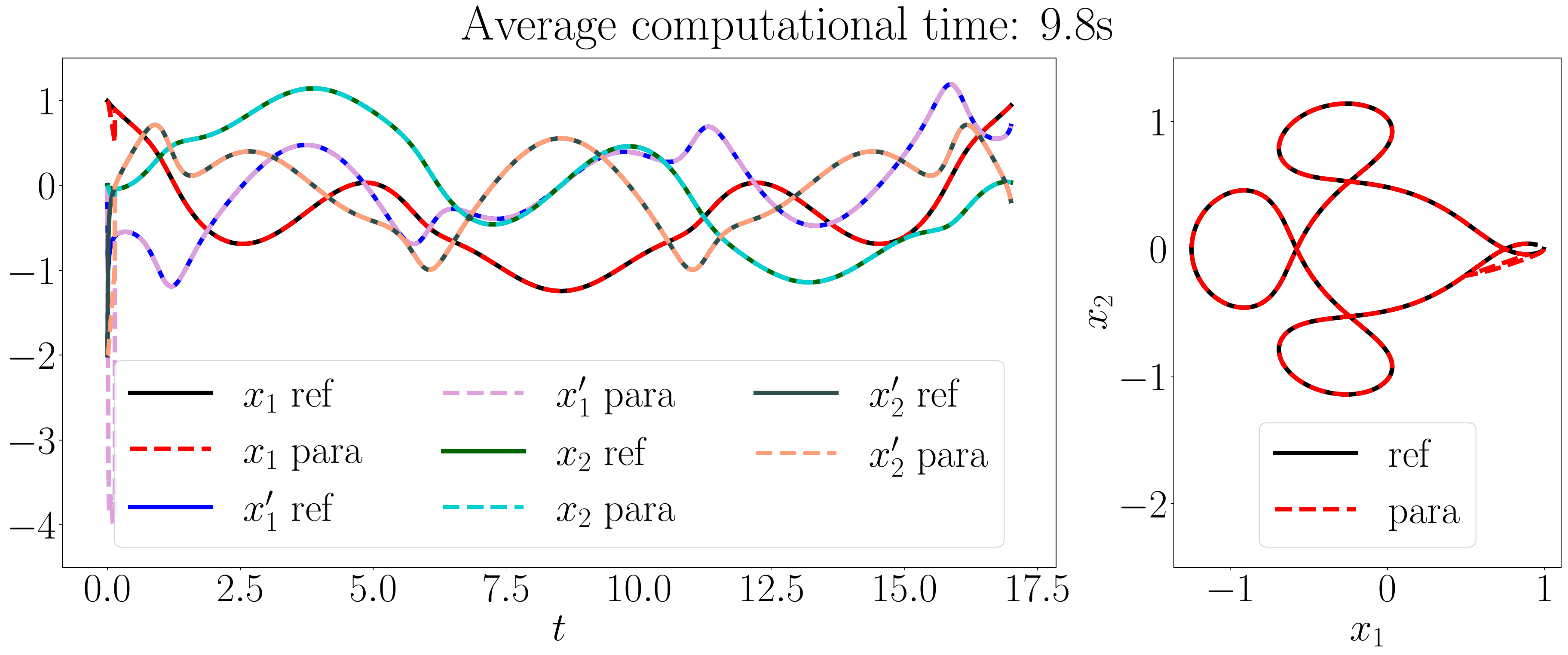}
\caption{{Arenstorf: Components of the
Hybrid Parareal solution (left), and the orbit of the initial condition (right).}}
\label{fig:paraArenstorf}
\vspace{-\baselineskip}
\end{figure}

This experiment serves to illustrate the benefits of using a Parareal-like correction of the neural network-based solution. Indeed, the approximate solution for short times does not accurately follow the correct trajectory. One possible remedy would be to restrict the step size $\Delta t$ as was done for the ROBER's problem. However, even for this larger time step choice, after just one step $\Delta t$, the Parareal correction resets the initial condition for the next interval bringing the solution back onto the stable orbit. Thus, not relying solely on a network-based solution allows us to compute an accurate solution for the later times, even though initially the solution departs the orbit.
\begin{table}[ht!]
    \centering
    \begin{tabular}{|l ||c|}
    \hline
         Timing breakdown  & RPNN\\
         \hline\hline
         Average cost coarse step in the zeroth iterate &  0.001912s\\
         \hline
         Average cost to produce the solution & 9.7957s \\
         \hline
    \end{tabular}
    \caption{{Arenstorf: Computational time for Hybrid Parareal using a single core.}}
    \label{tab:timeArenstorf}
    \vspace{-\baselineskip}
\end{table}

\subsection{Viscous Burgers' equation}\label{se:burgers}

Most of the systems considered up to now are low-dimensional. A natural way to test the method's performance on higher-dimensional systems is to work with spatial semi-discretizations of PDEs, where the mesh over which the spatial discretization is defined determines the system's dimension. We consider the one-dimensional Burgers' equation
\begin{equation}\label{eq:burgers}
\begin{cases}
\partial_t u\left(x,t\right) + u\left(x,t\right)\partial_x u\left(x,t\right) = \nu\partial_{xx}u\left(x,t\right),\,\quad x\in\Omega= \left[0,1\right],\\
u\left(x,0\right)=\sin{\left(2\pi x\right)},\,\quad x\in \Omega,\\
u\left(0,t\right)=u\left(1,t\right)=0,\,\quad t\geq 0.
\end{cases}
\end{equation}
In this section, we only report the results for the initial condition in Equation \eqref{eq:burgers}, but we include results for two more choices of initial conditions in Section \ref{app:burgers} of the supplementary material. All the experiments were run on five cores. In all tests we work with viscosity parameter $\nu = 1/50$, a uniform spatial grid of $51$ points in $\Omega=[0,1]$ and
coarse and fine step sizes $\Delta t = 1/50$ and $\delta t = 1/500$, respectively. The spatial semi-discretization with centered finite differences writes
\[
\begin{cases}
    \mathbf{u}'\left(t\right) = -\mathbf{u}\left(t\right)\odot \left(D_1\mathbf{u}\left(t\right)\right) + \nu D_2\mathbf{u}\left(t\right),\\
    \mathbf{u}\left(0\right)=\sin{\left(2\pi \x\right)}\in\mathbb{R}^{51},
\end{cases}
\]
where $\x=\begin{bmatrix} x_0 & x_1 & \dots & x_{50}\end{bmatrix}^{\top}$, $\x_i = i\Delta x$, $\Delta x = 1/50$, $i=0,\ldots,50$, $\odot$ is the component-wise product, and $D_1,D_2\in\mathbb{R}^{51\times 51}$ are the centered finite difference matrices of first and second order, respectively, suitably corrected to impose the homogeneous Dirichlet boundary conditions on $t\mapsto \mathbf{u}(t)$.
\begin{figure}[ht!]
\centering
\includegraphics[width=.7\textwidth]{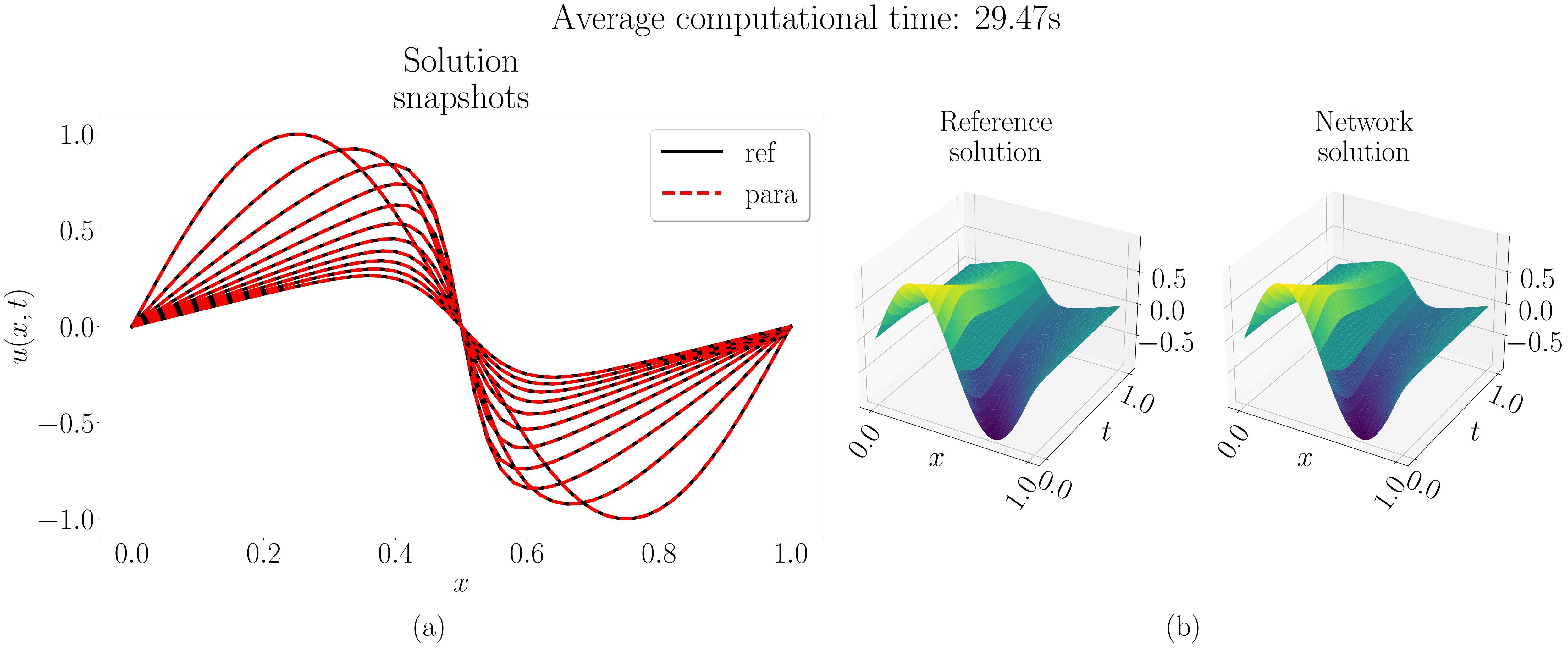}
\caption{{Burgers: Snapshots of the solution obtained with Hybrid Parareal (left), comparison of the solution surfaces between Hybrid Parareal and the fine integrator applied serially (right). Solution corresponding to $u_0(x)=\sin{(2\pi x)}$.}}
\label{fig:paraBurgersSingleWave}
\vspace{-\baselineskip}
\end{figure}

We report the qualitative behavior of the solutions in Figure \ref{fig:paraBurgersSingleWave}. Subfigure (a) tracks the solution at ten equally spaced time instants in the interval $[0,T=1]$. Subfigure (b) shows the solution surfaces obtained with the IE method on the left and the hybrid Parareal for one set of trained parameters on the right. We include the timings in Table \ref{tab:timeBurgersSingleWave}. We observe that the cost of the presented hybrid Parareal method grows with the dimensionality $d$ of the problem. However, we remark that for each of the $51$ components of the solution, we adopted only $H=5$ coefficients.

\begin{table}[ht!]
\centering
\begin{tabular}{|l ||c|}
\hline
Timing breakdown  & RPNN\\
\hline\hline
Average cost coarse step in the zeroth iterate &  0.2098s\\
\hline
Average cost to produce the solution & 29.4740s \\
\hline
\end{tabular}
\caption{{Burgers: Computational time for Hybrid Parareal using five cores, and initial condition $u_0(x)=\sin{(2\pi x)}$.}}
\label{tab:timeBurgersSingleWave}
\vspace{-\baselineskip}
\end{table}
\section{Conclusions and future extensions}\label{se:conclusions}
{In this manuscript, we proposed a hybrid parallel-in-time algorithm to solve initial value problems using a neural network as a coarse propagator within the Parareal framework. We derived an a-posteriori error estimate for generic neural network-based approximants. Based on these theoretical results we defined a hybrid Parareal algorithm involving RPNNs as coarse propagators which inherits the theoretical guarantees of the Parareal algorithm.}

We compared our hybrid Parareal solver based on RPNNs with one based on the flow map approach {on} the SIR problem. We demonstrated that our approach led to lower computational costs and no offline training phase. {We reserve the judgment of flow map performance.} However, we also tested it for other examples, {including} the Brusselator, where we noticed that the offline training phase can be very {intricate} because one has to first identify a forward invariant subset $\Omega$ of $\mathbb{R}^d$.

The most promising extension of this work is to include a mechanism allowing for time-adaptivity in the algorithm, i.e., for coarsening or {refinement of} the temporal grid based on the local behavior of the solution. {It would also be interesting to test our approach on higher-dimensional systems with high-performance computing hardware.}

\section*{Acknowledgments}
The authors would like to thank the Isaac Newton Institute for Mathematical Sciences for support and hospitality during the programme Mathematics of deep learning (MDL) when work on this paper was undertaken. This work was supported by: EPSRC grant number EP/R014604/1.
This work was partially supported by a grant from the Simons Foundation (DM).

\bibliography{references}

\begin{appendices}

\section{A-posteriori error estimate based on the defect}\label{se:convergenceDefect}
We now derive an alternative a-posteriori estimate for network-based approximate solutions based on defect control. 
\begin{lemma}\label{lem:linearODE}
Consider the initial value problem \eqref{eq:ode}, given by
\[
\begin{cases}
\xp\left(t\right) = \F\left(\x\left(t\right)\right),\\
\x\left(0\right) = \x_0,
\end{cases}
\]
where $\F:\mathbb{R}^d\to\mathbb{R}^d$ is continuously differentiable and admits a unique solution. Let $\y:\mathbb{R}\to\mathbb{R}^d$ satisfy
\begin{equation*}
\begin{cases}
\y'\left(t\right) = \F\left(\y\left(t\right)\right) + \bfd\left(t\right),\,\quad\bfd:\mathbb{R}\to\mathbb{R}^d,\\
\y\left(0\right) = \x_0.
\end{cases}
\end{equation*}
Then $\z(t):=\y(t)-\x(t)$ satisfies the linear differential equation
\begin{equation}\label{eq:linearODE}
\begin{cases}
\z'\left(t\right)=A\left(t\right)\z\left(t\right)+\bfd\left(t\right),\\
\z\left(0\right) = 0,
\end{cases}
\end{equation}
where
\[
A(t) = \int_0^1 \DF(\x(t)+s\z(t))\dd{s} 
\]
and $\DF$ is the Jacobian matrix of $\F$.
\end{lemma}
\begin{proof}
To prove the lemma, it suffices to highlight that
\[
\F\left(\y\left(t\right)\right)-\F\left(\x\left(t\right)\right) = \int_0^1\frac{d}{ds}\F\left(\x\left(t\right)+s\z\left(t\right)\right)\dd{s} = A\left(t\right)\z\left(t\right).
\]
\end{proof}
The solution to the linear problem \eqref{eq:linearODE} satisfies the following bound:
\begin{lemma}[Theorem 1 in \cite{soderlind1984nonlinear}]\label{lem:boundLogarithmic}
Let $\z(t)$ solve the initial value problem in \eqref{eq:linearODE}. Suppose that $\|\bfd(t)\|_2\leq \varepsilon$ for $t\geq 0$. Then,
\begin{align*}
\|\z(t)\|_2 &\leq \varepsilon \int_{0}^{t}\exp\left(\int_{s}^t \mu_2\left(A\left(\tau\right)\right)\dd{\tau}\right)\dd{s},
\end{align*}
where
\[
\mu_2(A) = \lambda_{\max}\left(\frac{A+A^{\top}}{2}\right)
\]
is the logarithmic $2-$norm of $A$.
\end{lemma}
For the proof of this lemma, see \cite[Theorem 10.6]{hairer1993solvingAPP}. 

As we are interested in solving \eqref{eq:ode}, we set $\y(t):=\ELM{\theta,t,0,{\x_0}}$ and introduce the defect function
\[
\bfd\left(t\right) := \dELM{\theta,t,0,{\x_0}}-\F\left(\ELM{\theta,t,0,{\x_0}}\right).
\]
We remark that the definition of $\bfd$ is of the same form as the loss \eqref{eq:loss}. If it was known that $\|\bfd(t)\|_2\leq \varepsilon$ for a tolerance $\varepsilon>0$ and all $t\in [0,\Delta t]$, then by Lemma \ref{lem:boundLogarithmic} we could conclude that
\begin{equation*}
\left\|\x\left(t\right)-\ELM{\theta,t,0,{\x_0}}\right\|_2 \leq \varepsilon\,\int_{0}^{t}\exp\left(\int_{s}^t \mu_2\left(A\left(\tau\right)\right)\dd{\tau}\right)\dd{s},\quad t\in \left[0,\Delta t\right].
\end{equation*}
Given that the solution $\x(t)$ is unknown, $A(\tau)$ and its logarithmic norm cannot be computed exactly. Thus, for a more practical error estimate, we introduce an assumption on the existence of a compact subset $\Omega\subset\mathbb{R}^d$ such that $\x(t)+s\z(t)\in \Omega$ for $(s,t)\in [0,1]\times [0,\Delta t]$. Then, we can proceed with the inequality chain as
\begin{equation}\label{eq:boundLogNorm}
\left\|\x\left(t\right)-\ELM{\theta,t,0,{\x_0}}\right\|_2 \leq \varepsilon\,\int_{0}^{t} e^{M(t-s)}\dd{s} = \varepsilon \frac{e^{Mt}-1}{M}, 
\end{equation}
where $M:=\max_{\z\in\Omega}\mu_2(\DF(\z))\in \mathbb R$. Note that the right-hand side of \eqref{eq:boundLogNorm} is nonnegative for all $t\geq 0$. In particular, \eqref{eq:boundLogNorm} implies that a neural network $\mathcal{N}_{\theta}$ can be employed to approximate the solution of \eqref{eq:ode} which is as accurate as a classical coarse solver $\coarse$ provided the norm of the defect $\|\bfd(t)\|_2$ is sufficiently small. 

\section{Bound on the norm of the sensitivity matrix}\label{app:boundNormJac}
In this appendix, we provide a practical bound for the norm of the Jacobian of the flow map of a vector field $\F$, assumed to be continuously differentiable with respect to the initial condition. For this, we differentiate the initial value problem \eqref{eq:ode}, given by
\begin{equation}\label{eq:odeFlow}
\begin{cases}
\frac{d}{dt}\phi_{\F}^{s,t}\left(\x_0\right) = \F\left(\phi_{\F}^{s,t}\left(\x_0\right)\right)\in\mathbb{R}^d,\\
\phi_{\F}^{s,s}\left(\x_0\right) = \x_0,
\end{cases}
\end{equation}
with respect to $\x_0$ and obtain
\begin{equation}\label{eq:variationalEquation}
\begin{cases}
\frac{d}{dt}\left(\frac{\partial \phi_{\F}^{s,t}\left(\x_0\right)}{\partial \x_0}\right) = \DF\left(\phi_{\F}^{s,t}\left(\x_0\right)\right)\frac{\partial \phi_{\F}^{s,t}\left(\x_0\right)}{\partial \x_0}\in\mathbb{R}^{d\times d},\\
\frac{\partial \phi_{\F}^{s,s}\left(\x_0\right)}{\partial \x_0} = I_{d},
\end{cases}
\end{equation}
where $I_{d}\in\mathbb{R}^{d\times d}$ is the identity matrix. Equation \eqref{eq:variationalEquation} is generally known as the variational equation of \eqref{eq:odeFlow}. This ODE is a non-autonomous linear differential equation in the unknown matrix $\partial_{\x_0}\phi_{\F}^{s,t}(\x_0)$. In practice, \eqref{eq:variationalEquation} should be solved jointly with \eqref{eq:odeFlow}. However, for the purpose of bounding the Euclidean norm $\|\partial_{\x_0}\phi_{\F}^{s,t}(\x_0)\|_2$, it is not necessary to solve them. Following \cite[Chapter 2]{desoer2009feedback}, we assume that $\phi^{s,t}_{\F}(\x_0)\in \Omega$ for  $\Omega\subset\mathbb{R}^d$ compact and all $0\leq s\leq t\leq \Delta t$. This is not a restrictive assumption on compact time intervals given the assumed regularity for $\F$. Then, one can get
\begin{align*}
\left\|\partial_{\x_0}\phi_{\F}^{s,t}\left(\x_0\right)\right\|_2&\leq \left\|\partial_{\x_0}\phi_{\F}^{s,s}\left(\x_0\right)\right\|_2\exp{\left(\int_s^t\mu_2\left(\DF\left(\phi_{\F}^{s,s'}\left(\x_0\right)\right)\right)\dd{s'}\right)} \\
&=\exp\left(\int_s^t\mu_2\left(\DF\left(\phi_{\F}^{s,s'}\left(\x_0\right)\right)\right)\dd{s'}\right)\leq \exp\left(M\Delta t\right),
\end{align*}
where $M=\max_{\z\in\Omega}\mu_2(D\F(\z))$. We conclude that the constant $\delta$ in the proof of Theorem \ref{thm:QuadAposteriori} can be set to $\exp(M\Delta t)$, with $M$ positive or negative depending on $\F$.

\section{The Jacobian matrix of the loss function}\label{app:jacobian}

In this subsection, we consider the loss function \eqref{eq:matrixEqn} and its gradient. Note that \eqref{eq:matrixEqn} can be expressed as \eqref{eq:loss} which in turn can be related to the solution of the non-linear matrix equation
\[
\dtX_{\theta}\left(\x,\Delta t\right) = \bfF\left(\tX_{\theta}\left(\x,\Delta t\right)\right).
\]
More explicitly, we have
\begin{align*}
\tX_{\theta}\left(\x,\Delta t\right)&=\bb{1}_C\x^{\top} + \left(\bfH-\overline{\bfH}\right)\theta,\,\,\bb{1}_C=\begin{bmatrix} 1 & \dots & 1\end{bmatrix}^{\top}\in\mathbb{R}^C,\\
\dtX_{\theta}\left(\x,\Delta t\right)&=\bfH'\theta.
\end{align*}
To minimize the loss function \eqref{eq:loss}, we need the Jacobian of the matrix-valued function $\mathbf{G}_{\theta}(\x,\Delta t) = \dtX_{\theta}(\x,\Delta t)-\bfF(\tX_{\theta}(\x,\Delta t))$. As $\mathbf{G}_{\theta}$ is a matrix-valued function with matrix inputs, we rely on the vectorization operator, denoted by $\mathrm{vec}$, using the machinery of matrix-calculus introduced, for example, in \cite{magnus2019matrix}. We hence compute $\frac{\partial \veco{\mathbf{G}_{\theta}\left(\x,\Delta t\right)}}{\partial\veco{\theta}} \in \mathbb{R}^{Cd\times Hd}$, given by
\begin{align*}
\frac{\partial \veco{\mathbf{G}_{\theta}\left(\x,\Delta t\right)}}{\partial\veco{\theta}} &= I_d\otimes \bfH' - \frac{\partial \veco{\bfF\left(\tX_{\theta}\left(\x,\Delta t\right)\right)}}{\partial\veco{\theta}} \\
&=I_d\otimes \bfH' - \frac{\partial \veco{\bfF\left(\bfX\right)}}{\partial \veco
{\bfX}}\Big\vert_{\bfX=\tX_{\theta}\left(\x,\Delta t\right)}\frac{\partial \veco{\tX_{\theta}\left(\x,\Delta t\right)}}{\partial \veco
{\theta}}\\
&= I_d\otimes \bfH' - \frac{\partial \veco{\bfF\left(\bfX\right)}}{\partial \veco
{\bfX}}\Big\vert_{\bfX=\tX_{\theta}\left(\x,\Delta t\right)}\left(I_d\otimes \left(\bfH-\overline{\bfH}\right)\right),
\end{align*}
where $I_d\in\mathbb{R}^{d\times d}$ is the identity matrix, $\otimes$ is the Kronecker product, and $\mathrm{vec}$ stacks the columns of the input matrix into a column vector. The Jacobian of $\bfF$ in the last line depends on the vector field $\F$, while the other terms do not. 

Most of the dynamical systems we consider in the numerical experiments in Section \ref{se:experiments} are of low dimension. For this reason, for all the cases but Burgers' equation, we assemble the Jacobian case by case, following this construction. For Burgers' equation, we instead implement it as a linear operator, specifying its action and the action of its transpose onto input vectors. For the Burgers' equation, we have
\[
\F\left(\mathbf{u}\right) = -\mathbf{u}\odot \left(\D_1\mathbf{u}\right) + \nu \D_2\mathbf{u}\in\R^d,
\]
and hence $\bfF(\bfX) = -\bfX\odot (\bfX \D_1^{\top}) + \nu\bfX \D_2^{\top}\in\R^{C\times d}$.
This expression implies that
\[
\frac{\partial \veco{\bfF\left(\bfX\right)}}{\partial \veco
{\bfX}} = -\mathrm{diag}\left(\veco{\bfX \D_1^{\top}}\right) - \mathrm{diag}\left(\veco{\bfX}\right)\left(\D_1\otimes I_C\right)+\nu \D_2\otimes I_C.
\]
\section{Details on the network for the flow map approach}\label{se:FlowMapApproach}
In this section, we provide details on the network for the flow map approach required for the comparison of the training costs presented in Table \ref{tab:timeSIR}. The network used for the coarse propagator is based on the parametrization
\begin{align*}
\z:=\begin{bmatrix}\x_0^{\top},t\end{bmatrix}^{\top}&\mapsto \tanh\left(\mathbf{A}_0\z+\mathbf{a}_0\right)=:\mathbf{h}_1\in\mathbb{R}^{10},\\
\mathbf{h}_{\ell}&\mapsto  \tanh\left(\mathbf{A}_{\ell}\mathbf{h}_{\ell} + \mathbf{a}_{\ell}\right)=:\mathbf{h}_{\ell+1}\in\mathbb{R}^{10},\,\,\ell=1,\cdots,4,\\
\mathbf{h}_{5}&\mapsto \x_0 + \left(1-e^{-t}\right)\mathbf{P}\mathbf{h}_5=:\ELM{\theta,t,0,{\x_0}}\in\mathbb{R}^3,
\end{align*}
where $\theta = \{\mathbf{A}_{\ell},\mathbf{a}_{\ell},\mathbf{P}\}_{\ell=0}^4$. To train the network, implemented with PyTorch, we use the Adam optimizer for $10^5$ epochs, with each epoch consisting of minimizing the ODE residual over $500$ different randomly sampled collocation points $(t^i,\x_0^i)\in [0,1]\times [0,1]^3$. 
\section{Experiment for Brusselator's equation}

This section collects numerical experiments for the Brusselator, which is a system of two scalar differential equations modeling a chain of chemical reactions \cite{ault2003dynamics}. The equations write
\begin{equation}
\begin{cases}
x_1'\left(t\right) = A + x_1^2\left(t\right)x_2\left(t\right) - \left(B+1\right)x_1\left(t\right),\\
x_2'\left(t\right) = Bx_1\left(t\right)-x_1^2\left(t\right)x_2\left(t\right),\\
\x\left(0\right) = \begin{bmatrix} 0 & 1 \end{bmatrix}^{\top},
\end{cases}
\end{equation}
where we choose the parameters $A = 1$, $B=3$. In this setting, one can prove to have a limit cycle in the dynamics. We simulate this system on the time interval $[0,T=12]$, with a fine timestep $\delta t = T/640$ and a coarse one of size $\Delta T = T/32$.
\begin{figure}[ht!]
\centering
\includegraphics[width=.7\textwidth]{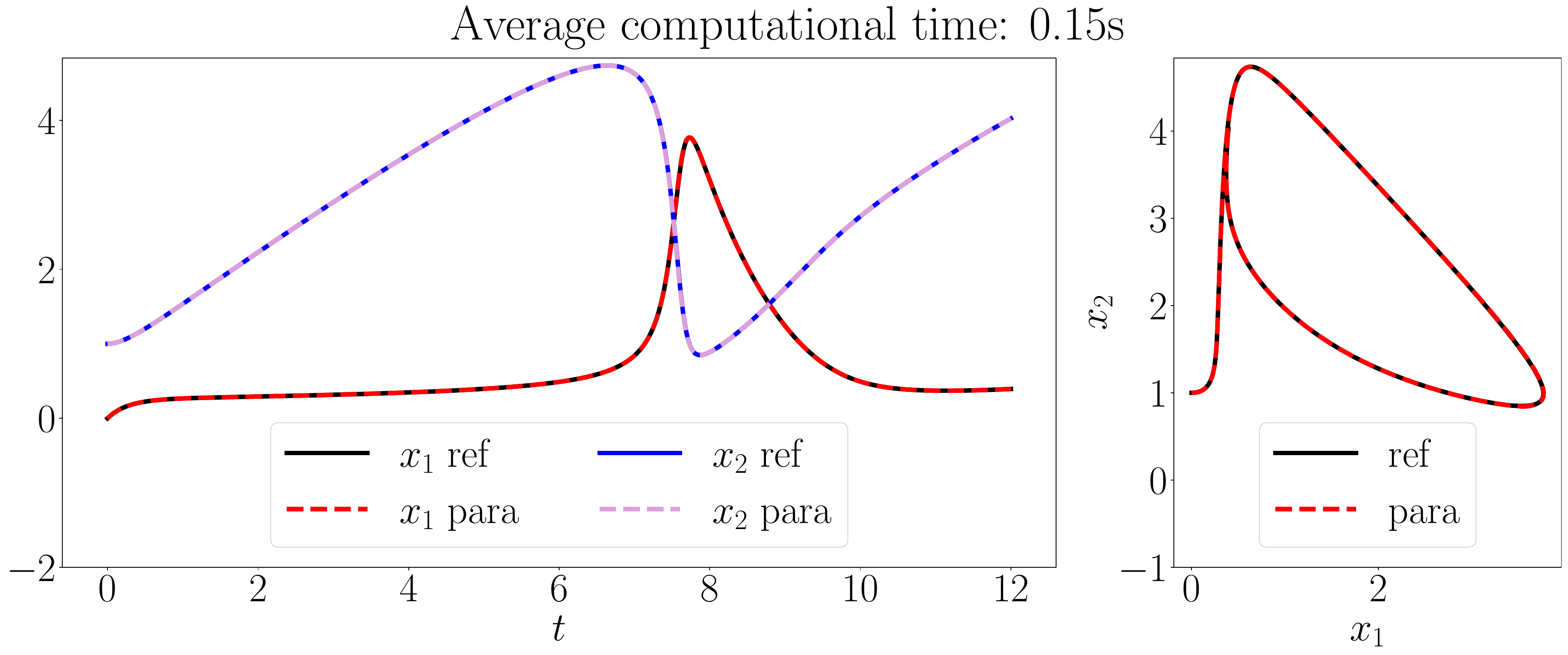}
\caption{Qualitative accuracy of the predicted solutions using Parareal with RPNN.}
\label{fig:paraBrusselator}
\vspace{-\baselineskip}
\end{figure}
We repeat the simulation 100 times, reporting the average cost of one coarse timestep in Table \ref{tab:timeBrusselator}, together with the average total cost of the hybrid Parareal solver. Figure \ref{fig:paraBrusselator} shows the approximate solution and a reference solution. We also remark that, as desired, the hybrid method recovers the limit cycle. 
\begin{table}[ht!]
    \centering
    \begin{tabular}{|l||c|}
    \hline
         Timing breakdown & RPNN\\
         \hline\hline
         Average cost coarse step in the zeroth iterate &  0.001012s\\
         \hline
         Average cost to produce the solution & 0.1469s \\
         \hline
    \end{tabular}
    \caption{Brusselator: Computational time for Hybrid Parareal using a single core.}
    \label{tab:timeBrusselator}
    \vspace{-\baselineskip}
\end{table}

\section{Additional experiments for Burgers' equation}\label{app:burgers}
\begin{figure}[ht!]
\centering
\includegraphics[width=.7\textwidth]{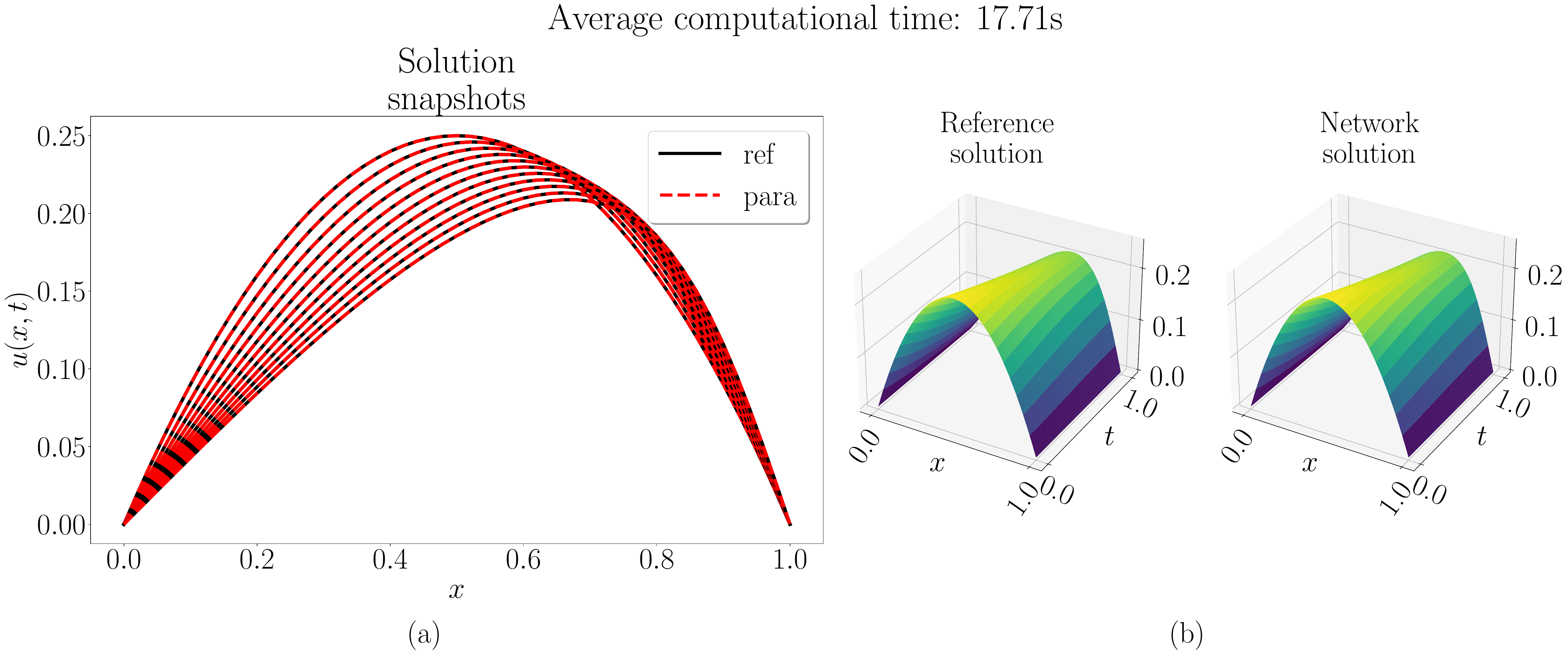}
\caption{Burgers: Snapshots of the solution obtained with Hybrid Parareal (left), comparison of the solution surfaces between Hybrid Parareal and the fine integrator applied serially (right). Solution corresponding to $u_0(x)=x(1-x)$.}
\label{fig:paraBurgersQuadratic}
\vspace{-\baselineskip}
\end{figure}

\begin{table}[ht!]
    \centering
    \begin{tabular}{|l||c|}
    \hline
         Timing breakdown & RPNN\\
         \hline\hline
         Average cost coarse step in the zeroth iterate &  0.1695s\\
         \hline
         Average cost to produce the solution & 17.7069s \\
         \hline
    \end{tabular}
    \caption{Burgers: Computational time for Hybrid Parareal using five cores, and initial condition $u_0(x)=x(1-x)$.}
    \label{tab:timeBurgersQuadratic}
    \vspace{-\baselineskip}
\end{table}

\begin{figure}[ht!]
\centering
\includegraphics[width=.7\textwidth]{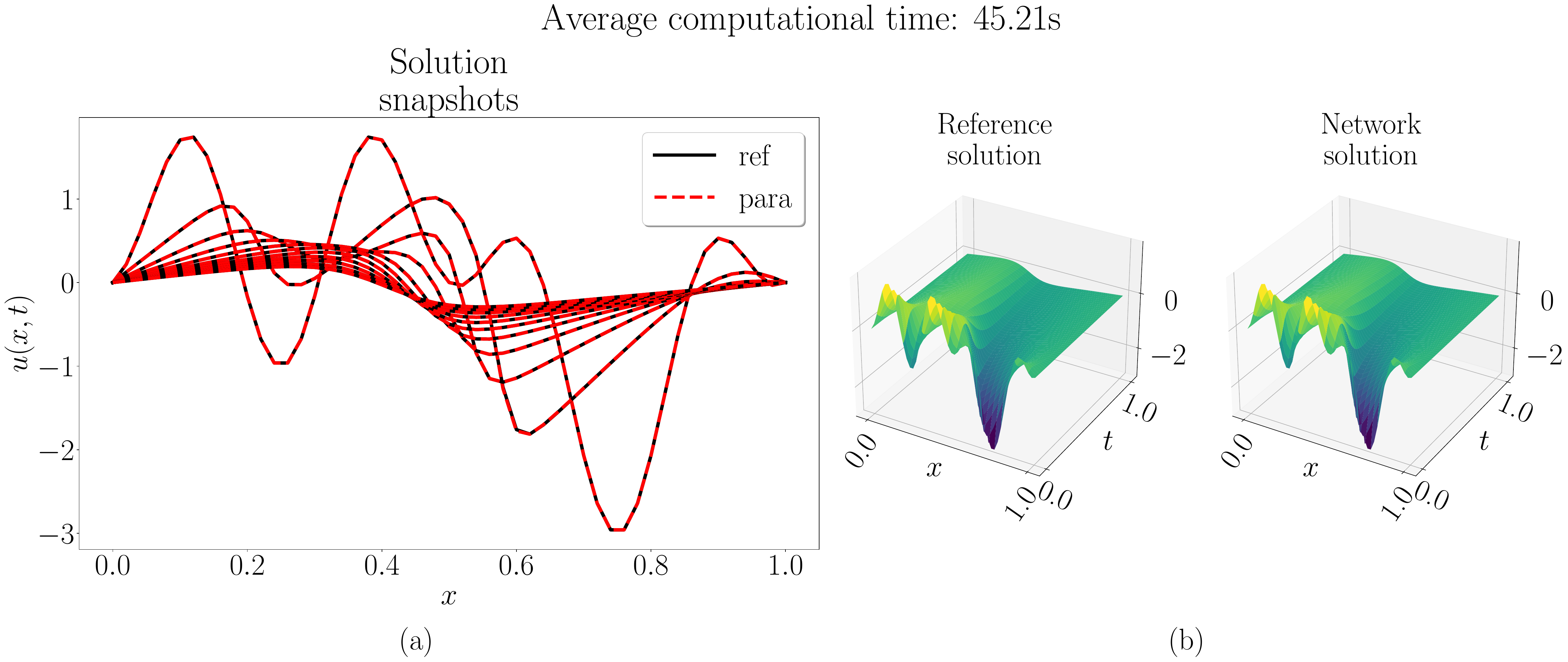}
\caption{Burgers: Snapshots of the solution obtained with Hybrid Parareal (left), comparison of the solution surfaces between Hybrid Parareal and the fine integrator applied serially (right). Solution corresponding to $u_0(x)=\sin{(2\pi x)} + \cos{(4\pi x)} - \cos{(8\pi x)} $.}
\label{fig:paraBurgersSumOfWaves}
\vspace{-\baselineskip}
\end{figure}

\begin{table}[ht!]
    \centering
    \begin{tabular}{|l||c|}
    \hline
         Timing breakdown & RPNN\\
         \hline\hline
         Average cost coarse step in the zeroth iterate &  0.3356s\\
         \hline
         Average cost to produce the solution & 45.2056s \\
         \hline
    \end{tabular}
    \caption{Burgers: Computational time for Hybrid Parareal using five cores, and initial condition $u_0(x)=\sin{(2\pi x)} + \cos{(4\pi x)} - \cos{(8\pi x)}$.}
    \label{tab:timeBurgersSumWaves}
    \vspace{-\baselineskip}
\end{table}
In this section, we report the simulation results for the Burgers' equation with two more initial conditions. The setup of the network and the partition of the time domain are the same as for the initial condition included in Section~\ref{se:burgers}. In Figure \ref{fig:paraBurgersQuadratic}, we work with the initial condition $u_0(x)=x(1-x)$, while in Figure \ref{fig:paraBurgersSumOfWaves} with $u_0(x)=\sin{(2\pi x)} + \cos{(4\pi x)} - \cos{(8\pi x)}$. The timings are included in Tables \ref{tab:timeBurgersQuadratic} and \ref{tab:timeBurgersSumWaves}, respectively. As expected, the time to obtain the full solution grows with the complexity of the initial condition. Indeed, there are about 10 seconds of difference between the fastest, corresponding to the quadratic initial condition in Figure \ref{fig:paraBurgersSumOfWaves}, to the second fastest, the one with $u_0(x)=\sin{(2\pi x)}$, and the slowest in Figure \ref{fig:paraBurgersSumOfWaves}. The reason behind this observed behavior is that, for more complicated solutions, the coarse predictions need to be corrected with the Parareal correction step more often, and the optimization problems to solve to get the coarse propagator get more expensive.

\end{appendices}

\end{document}

%% file: shared.tex
\usepackage[utf8]{inputenc}
\usepackage{xcolor}
\usepackage{algpseudocode}
\usepackage{dirtytalk}
\usepackage{algorithm}
\usepackage{amsfonts,amsmath,amssymb,amsthm}
\usepackage{lmodern}
\usepackage[title]{appendix}%
\usepackage{manyfoot}%

\usepackage{hyperref}

\newtheorem{theorem}{Theorem}
\newtheorem{proposition}{Proposition}
\newtheorem{lemma}{Lemma}
\newtheorem{assumption}{Assumption}

\def\R{\mathbb{R}}

\newcommand{\veco}[1]{\mathrm{vec}\left({#1}\right)}

\newcommand{\x}{\mathbf{x}}
\newcommand{\y}{\mathbf{y}}
\newcommand{\z}{\mathbf{z}}

\newcommand{\bb}[1]{\mathbf{#1}}

\newcommand{\tx}{\mathbf{\tilde{x}}}

\newcommand{\xp}{{\x}'}
\newcommand{\yp}{{\y}'}
\newcommand{\dtx}{{\tx}'}
\newcommand{\e}{\bb{e}}

\newcommand{\fine}[1][\Delta t]{\varphi^{{#1}}_F}
\newcommand{\coarse}[1][\Delta t]{\varphi^{{#1}}_C}

\newcommand{\ELM}[1]{\ELMhelper#1}
\def\ELMhelper#1,#2,#3,#4{
    \mathcal{N}_{#1}\left(#2;#4\right)
}

\newcommand{\dELM}[1]{\dELMhelper#1}
\def\dELMhelper#1,#2,#3,#4{
    \mathcal{N}_{#1}'\left(#2;#4\right)
}

\newcommand{\qr}[1]{\inthelper#1}
\def\inthelper#1,#2,#3{
\mathcal{I}_p\left(#1;#2,#3\right)
}

\newcommand{\tX}{\mathbf{\tilde{X}}}

\newcommand{\bfF}{\bb{F}}
\newcommand{\bfH}{\bb{H}}

\newcommand{\bfX}{\mathbf{X}}
\newcommand{\bfY}{\mathbf{Y}}
\newcommand{\D}{\mathbf{D}}

\newcommand{\dtX}{{\tX}'}

\newcommand{\bfa}{\bb{a}}
\newcommand{\bfb}{\bb{b}}
\newcommand{\bfd}{\mathbf{d}}

\newcommand{\F}{\mathcal{F}}

\newcommand{\DF}{\operatorname{D}\F}
\newcommand{\dd}[1]{\mathrm{d}{#1}}

\newcommand{\inv}{\left(\bfH'\right)^{-1}}